\documentclass[10pt,reqno]{amsart}

\usepackage[margin=1in]{geometry}
\usepackage{amsmath,amssymb,amsthm,mathtools}
\usepackage{mathrsfs}
\usepackage{enumitem}
\usepackage{microtype}
\usepackage{graphicx}
\usepackage{float}
\usepackage{tikz}
\usetikzlibrary{arrows.meta,positioning}
\numberwithin{equation}{section}

\usepackage[dvipsnames]{xcolor}
\usepackage[colorlinks=true,linkcolor=blue,urlcolor=blue,citecolor=PineGreen]{hyperref}
\usepackage{sidecap}
\usepackage[normalem]{ulem}

\newtheorem{theorem}{Theorem}[section]
\newtheorem{proposition}{Proposition}[section]
\newtheorem{lemma}{Lemma}[section]
\newtheorem{corollary}{Corollary}[section]
\newtheorem{remark}{Remark}[section]
\theoremstyle{definition}

\newcommand{\R}{\mathbb R}
\newcommand{\dd}{\mathrm d}
\newcommand{\dx}{\mathrm{d}x}
\newcommand{\dpp}{\mathrm{d}p}
\newcommand{\dt}{\mathrm{d}t}
\newcommand{\ds}{\mathrm{d}s}
\newcommand{\dy}{\mathrm{d}y}
\newcommand{\dr}{\mathrm{d}r}

\newcommand{\be}{B^{\varepsilon}}
\newcommand{\kappah}{\kappa_{\hbar}}
\newcommand{\e}{\varepsilon}
\newcommand{\one}{\mathbf 1}

\newcommand{\calC}{\mathcal C}

\newcommand{\calE}{\mathcal E}
\newcommand{\calF}{\mathcal F}

\newcommand{\bq}{\begin{equation}}
\newcommand{\eq}{\end{equation}}
\newcommand{\lt}{\left}
\newcommand{\rt}{\right}
\newcommand{\pa}{\partial}

\usepackage[colorinlistoftodos]{todonotes}

\usepackage{cancel}

\title{Global mild solutions and the semiclassical limit for the Fermi--Dirac BGK model}

\author[Choi]{Young-Pil Choi}
\address[Young-Pil Choi]{\newline Department of Mathematics\newline
Yonsei University, 50 Yonsei-Ro, Seodaemun-Gu, Seoul 03722, Republic of Korea}
\email{ypchoi@yonsei.ac.kr}

\author[Hwang]{Byung-Hoon Hwang}
\address[Byung-Hoon Hwang]{\newline Department of Mathematics Education\newline
	Sangmyung University, 20 Hongjimun 2-gil, Jongno-Gu, Seoul 03016, Republic of Korea}
\email{bhhwang@smu.ac.kr}

 \author[Song]{Sihyun Song}
 \address[Sihyun Song]{\newline Department of Mathematics\newline
 Yonsei University, 50 Yonsei-Ro, Seodaemun-Gu, Seoul 03722, Republic of Korea}
 \email{ssong@yonsei.ac.kr}

\begin{document}

\allowdisplaybreaks

\date{\today}


\begin{abstract}
We prove global existence of mild solutions to the spatially inhomogeneous Fermi--Dirac BGK equation for arbitrary Pauli-admissible initial data with finite mass and kinetic energy, allowing both vacuum and locally saturated zero-temperature states. The construction relies on a moment-compatible regularization of the local equilibrium that preserves the Pauli bound and uniform moment control while remaining consistent up to the saturation boundary. The solutions satisfy mass, momentum, and energy conservation and, under an additional finite spatial second-moment assumption, an entropy-gap H-theorem on the full Pauli interval.  Under uniform moment and entropy bounds, we further prove that, after extraction of a subsequence, quantum mild solutions converge strongly in phase space, uniformly on bounded time intervals, to a mild solution of the classical BGK equation.
\end{abstract}

\maketitle \centerline{\date}

\tableofcontents

%
%
%
%
%
%
\section{Introduction}
%
%
%
%
%
%
\subsection{The Fermi--Dirac BGK equation}

In this paper we study the Fermi--Dirac BGK equation
\bq \label{eq:FD-BGK}
\pa_t f+p\cdot\nabla_x f=\calF(f)-f,  \quad  (t,x,p)\in(0,\infty)\times\R^3\times\R^3,
\eq
where $\hbar>0$ is the quantum parameter and the Pauli exclusion principle is represented by
\[
 0\leq f\leq\frac1\hbar .
\]
Here $\calF(f)$ denotes the Fermi--Dirac local equilibrium associated with the quantum parameter $\hbar$. The collision operator in \eqref{eq:FD-BGK} is a relaxation-time approximation of the fermionic quantum Boltzmann equation associated with Fermi--Dirac statistics \cite{UU33}.

For a nonnegative distribution $f$, set
\[
 N_f(t,x):=\int_{\R^3}f(t,x,p)\,\dpp,  \quad  P_f(t,x):=\int_{\R^3}pf(t,x,p)\,\dpp,  \quad  E_f(t,x):=\int_{\R^3}|p|^2f(t,x,p)\,\dpp.
\]
The case $N_f(t,x)=0$ is immediate: nonnegativity implies $f(t,x,\cdot)=0$ almost everywhere, and we set $\calF(f)(t,x,\cdot)=0$.  Thus, we describe the local equilibrium on the non-vacuum set $\{N_f>0\}$.

Define
\[
 u_f:=\frac{P_f}{N_f},  \quad  \calE_f:=E_f-\frac{|P_f|^2}{N_f}.
\]
The quantity $\calE_f$ is the kinetic energy relative to the local mean velocity $u_f$, since
\[
 \calE_f  =  \int_{\R^3}|p-u_f|^2f(t,x,p)\,\dpp. 
 \]  
As will be shown below, the Pauli constraint implies the sharp inequality
\bq \label{eq:pa-b}
 \calE_f\geq \kappah N_f^{\frac53},  \quad  \kappah:=  \frac{3^{\frac53}}{5(4\pi)^{\frac23}}\hbar^{\frac23}.
\eq
Accordingly, the local equilibrium appearing in \eqref{eq:FD-BGK} is given by
\bq \label{eq:F-s}
 \calF(f)(t,x,p)
 =
 \begin{cases}
 \displaystyle  \frac{1}{  \exp \lt(  a_f(t,x)|p-u_f(t,x)|^2+c_f(t,x)  \rt)+\hbar},  &  \calE_f(t,x) >\kappah N_f(t,x)^{\frac53},  \\[3mm]
 \displaystyle  \frac1\hbar  \one_{\{|p-u_f(t,x)|\leq R_f(t,x)\}},  &  \calE_f(t,x) = \kappah N_f(t,x)^{\frac53},
 \end{cases}
\eq
where
\[
 R_f:=\lt(\frac{3\hbar N_f}{4\pi}\rt)^{\frac13}.
\]
The first line of \eqref{eq:F-s} is the finite-temperature Fermi--Dirac equilibrium, whereas the second line is its saturated zero-temperature endpoint.  We next describe how the parameters $a_f$ and $c_f$ in the first case are determined by the macroscopic moments and why the threshold in \eqref{eq:pa-b} is the natural one.

For $N>0$, set
\[
 B(N,P,E):=  \frac{N}{(E-|P|^2/N)^{\frac35}},  \quad  B_f:=B(N_f,P_f,E_f).
\]
For $k\geq0$, define
\[
 \beta_k^{\hbar}(c):=  \int_{\R^3}  \frac{|p|^k}{e^{|p|^2+c}+\hbar}\,\dpp,  \quad  \beta^{\hbar}(c):=  \frac{\beta_0^{\hbar}(c)}{\beta_2^{\hbar}(c)^{\frac35}}.
\]
Then, the explicit relation between $B$ and $\beta^{\hbar}$ may be derived in the following way. Writing
\[
 F_{a,c,u}(p):=  \frac1{e^{a|p-u|^2+c}+\hbar},  \quad a>0,
\]
the change of variables $q=\sqrt a\,(p-u)$ gives
\[
 N_{F_{a,c,u}} =a^{-\frac32}\beta_0^{\hbar}(c), \quad P_{F_{a,c,u}} =uN_{F_{a,c,u}},\quad  E_{F_{a,c,u}}  -\frac{|P_{F_{a,c,u}}|^2}{N_{F_{a,c,u}}} =a^{-\frac52}\beta_2^{\hbar}(c).
\]
Eliminating $a$ from these identities yields
\[
 B(  N_{F_{a,c,u}},  P_{F_{a,c,u}},  E_{F_{a,c,u}})  =  \beta^{\hbar}(c).
\]

The function $\beta^{\hbar}$ is strictly decreasing from $\R$ onto $(0,\beta^{\hbar}(-\infty))$, where
\[
 \beta^{\hbar}(-\infty)  =  \frac{5^{\frac35}(4\pi)^{\frac25}}  {3\hbar^{\frac25}}.
\]
See Lemma \ref{lem:beta} below.  Hence, whenever
\[
 0<B_f<\beta^{\hbar}(-\infty),
\]
there exists a unique $c_f\in\R$ satisfying
\[
 \beta^{\hbar}(c_f)=B_f,  \quad  a_f:=  \beta_0^{\hbar}(c_f)^{\frac23}N_f^{-\frac23}.
\]
The coefficient $a_f$ is the nondimensional inverse-temperature scale, whereas $c_f$ is a chemical-potential-type parameter.

The relation between the range of $B$ and the two cases in \eqref{eq:F-s} follows from the sharp Pauli inequality. Namely, we show in Lemma \ref{lem:sharp-pauli} that
\[
 \calE_f\geq\kappah N_f^{\frac53},  \quad  \kappah=\beta^{\hbar}(-\infty)^{-\frac53},
\]
and also
\[
 \calE_f>\kappah N_f^{\frac53}  \quad\Longleftrightarrow\quad  B_f<\beta^{\hbar}(-\infty),
\]
thus in this case $c_f\in \R$ is well-defined. This means $\calF(f)$, as corresponding to the finite-temperature regime (the first line of \eqref{eq:F-s}), is also well-defined. At the endpoint
\[
 \calE_f=\kappah N_f^{\frac53}  \quad\Longleftrightarrow\quad  B_f=\beta^{\hbar}(-\infty),
\]
we show that the appropriate (and unique) admissible Fermi--Dirac distribution is the saturated one
\[
 \frac1\hbar  \one_{\{|p-u_f|\leq R_f\}},
\]
that is, the second line of \eqref{eq:F-s}. This distribution is, at least formally, the zero-temperature endpoint of the finite-temperature Fermi--Dirac family.  Indeed, writing
\[
 a_f|p-u_f|^2+c_f  =  a_f\lt(|p-u_f|^2-\mu_f\rt),  \quad  \mu_f:=-\frac{c_f}{a_f},
\]
one sees that, as the inverse-temperature scale $a_f$ becomes large, the transition across the level $|p-u_f|^2=\mu_f$ becomes sharp.  We refer to this resulting profile (the second line of \eqref{eq:F-s}) as the zero-temperature Fermi--Dirac distribution, or the saturated Fermi--Dirac distribution.

The intrinsic way to define both Fermi--Dirac distributions in \eqref{eq:F-s} is to use the quantum entropy. Define the entropy density
\[
 \eta_\hbar(z)  :=  z\log z  +\frac1\hbar(1-\hbar z)\log(1-\hbar z),  \quad  0\leq z\leq\frac1\hbar ,
\]
with the convention $0\log0=0$. Proposition \ref{prop:vari} shows that, for every non-vacuum moment triple $(N,P,E)$ arising from a Pauli-admissible distribution, the corresponding Fermi-Direc distribution $\calF[N,P,E]$ is the unique minimizer of
\[
 \int_{\R^3}\eta_\hbar(g)\,\dpp
\]
under the Pauli constraint and the prescribed mass, momentum, and energy. Related quantum entropy optimization problems are discussed in \cite{EMV05}. As a natural consequence of these constraints, it follows that for both of the regimes in \eqref{eq:F-s}, the local Fermi-Dirac distribution satisfies \textit{exact moment matching}:
\begin{equation} \label{eq/ emm}
 \int_{\R^3}(1,p,|p|^2)  \lt(\calF(f)-f\rt) \dpp=0.
\end{equation}
This can of course also be checked through direct computations.

We now review some previous related mathematical literature. For the classical BGK equation, global existence and stability based on compactness of the macroscopic fields, velocity-moment estimates, and entropy were established in \cite{Per89}. Weighted $L^\infty$ estimates and uniqueness were subsequently obtained in \cite{PP93}. Smooth solutions near a global Maxwellian and their large-time behavior were later studied in \cite{Yun10}. For the ES--BGK model, global theories are available for both perturbative and for large initial data \cite{HY19,Yun15JDE,Yun15SIAM}; multicomponent models and the large-amplitude problem have also been addressed in \cite{BKPY23,BKLY26}.

For the Fermi--Dirac BGK model, the available literature on the Cauchy problem is primarily perturbative. Global regular solutions are provided in \cite{BY20} under the assumption that the initial data is sufficiently close to a global Fermi--Dirac equilibrium. Stationary mild solutions in a slab are constructed in \cite{BY20S} under conditions that rule out a transition to the saturated local state (the zero-temperature regime in \eqref{eq:F-s}). Variants of quantum BGK models have also been investigated. Nouri \cite{Nou08} proved existence of bounded measure solutions for a stationary Bose--Einstein BGK model in a slab. Also, a multi-species BGK model associated with the Uehling--Uhlenbeck equation, including the well-definedness of its implicitly determined equilibrium parameters, conservation laws, and entropy structures, was developed in \cite{BKPY21}.

Related quantum relaxation models have been treated in analytic or perturbative regimes \cite{Bra19,Bra20}. We also mention that global near-equilibrium theories are available for quantum Boltzmann equations \cite{BJY21,OW22}; see also \cite{JZ25}. Spatially inhomogeneous weak-solution theories for the Boltzmann--Fermi--Dirac equation are available in the whole-space, general-domain, and periodic settings \cite{All10,Dol94,Lu08}. For the Landau--Fermi--Dirac equation we refer to \cite{Sam24p} for existence of global weak solutions.

In any of the models above, the zero-temperature saturated branch of the local equilibrium \eqref{eq:F-s} is a familiar feature in their mathematical studies. Let us first mention that for the spatially homogeneous Boltzmann--Fermi--Dirac equation, the equilibrium classification contains both finite-temperature Fermi--Dirac distributions and saturated characteristic functions of momentum balls \cite{Lu01}; stability and strong convergence for this dynamics are studied in \cite{LW03}. Recent results, in the case of hard-potentials, further develop the homogeneous theory and explicitly distinguish the saturated equilibrium case \cite{AP25p}. Saturated equilibria and the degeneracy associated with the Pauli upper bound also play an important role for the Landau--Fermi--Dirac equation \cite{ABDL22,BL04,GGZ22}. It is in particular worth emphasizing the recent work \cite{GH26}, where global smooth solutions are obtained for this equation, under suitable decay assumptions on the initial data. The idea there is that these bounds can be propagated to show suitable macroscopic bounds for the solution, which in turn lead to mass spreading and thus desaturation. This last step is crucial as the diffusion for the Landau--Fermi--Dirac equation degenerates at saturated states. Finally, for the spatially inhomogeneous Fermi--Dirac BGK equation under study here, we emphasize that it is unclear whether a global large-data theory for mild solutions can be established if we permit the local Fermi--Dirac distribution $\calF(f)$ to interchange between the zero and finite temperature regimes. This is one of the results that we establish in this paper: although we cannot rule out saturation under merely the physical assumptions on the initial data, we are nevertheless able to show that a mild solution exists which is continuous in time with respect to $L^1$. 

A separate line of work investigates semiclassical limits of fermionic kinetic equations. This is the second goal of the study here. For the spatially homogeneous quantum Boltzmann equation, the semiclassical limit has been justified first at the level of weak convergence and subsequently through quantitative asymptotic expansions \cite{HLP21,HLPZ24}.  In the spatially inhomogeneous setting, the semiclassical limit from the Coulomb Landau--Fermi--Dirac equation to a renormalized solution of the classical Landau equation (with defect measure) has recently been established in \cite{Sam26}. For the model considered in this paper, to the best of our knowledge, a rigorous semiclassical limit from the spatially inhomogeneous Fermi--Dirac BGK equation to the classical BGK equation has not previously been established. We note that the main obstacle is to investigate the limit of the local Fermi--Dirac distribution \eqref{eq:F-s} as the quantum parameter tends to zero. This task is not so trivial as all of the coefficients of $\calF(f)$ depend implicitly on the quantum parameter as well as the solution itself; another potential obstacle is that the zero-temperature state in \eqref{eq:F-s} converges formally to a Dirac measure as $\hbar\downarrow 0$, which is undesirable since what we expect is for $\calF(f)$ to approximate the local Maxwellian as $\hbar\downarrow 0$. We therefore rigorously prove that the zero-temperature local equilibria become ruled out as $\hbar\downarrow 0$, as long as physical assumptions are imposed on the initial data (see the proof of Theorem \ref{thm:semi}). 

%
%
%
%
%
%

\subsection{Main results}

Our first result assumes Pauli-admissible initial data with finite mass and kinetic energy. We construct global mild solutions that allow vacuum and permit the local moments to reach the sharp Pauli boundary corresponding to a zero-temperature equilibrium. To the best of our knowledge, this is the first global large-data existence result for this equation covering the full range of Pauli-admissible local moments. It is in particular worth noting that in the available literature \cite{BY20,BY20S}, enough assumptions were provided on the initial data in order to ensure no saturation occurs. In contrast, the setting here allows both vacuum and saturated states in the solution class.  Consequently, neither a positive lower bound on the local density nor a uniform separation from the Pauli boundary is needed.
 
The main difficulty is to control the nonlinear local equilibrium simultaneously near vacuum and near saturation. The bulk velocity $P_f/N_f$ becomes singular as the density vanishes, whereas the finite-temperature parametrization degenerates as the moment triple approaches the sharp Pauli boundary. We first characterize the admissible moment region by the sharp Pauli inequality and define the equilibrium variationally on this entire region, including its saturated boundary.

In resolution of the Cauchy problem, it turns out that constructing a regularized model corresponding to \eqref{eq:FD-BGK} is surprisingly delicate. The key ingredient is constructing a \textit{moment-compatible} regularization of the local equilibrium \eqref{eq:F-s}; see Lemma \ref{lem:unif-e}. Merely bounding the macroscopic parameters as in \cite{KS25} does not resolve the issue: the approximation must simultaneously remain Pauli-admissible, become Lipschitz for each fixed regularization parameter, and retain moment bounds that are uniform as the regularization is removed. We achieve this by coupling the regularization of the density and bulk velocity with a range truncation of the degeneracy parameter and a suitable interpolation near vacuum. At non-vacuum points, the resulting profiles remain within the finite-temperature Fermi--Dirac family. With this choice of regularization, we prove in Lemma \ref{lem:unif-e} that this one satisfies
\[
 \int_{\R^3}(1+|p|^2)\calF^\e(f)(p)\,\dpp \leq C_\hbar(N_f+E_f),
\]
with $C_\hbar$ independent of $\e$. In particular, the approximate local Fermi--Dirac distribution is \textit{moment-compatible} with the original one; this uniform moment estimate is crucial in the compactness arguments as $\e\to 0$. It is also worth emphasizing that the regularization is consistent even for converging moment triples whose limits lie on the saturation boundary; see Proposition \ref{prop:consist}.

Due to the fact that the coefficients of the local Fermi--Dirac distributions depend also on the second moments of the solution, we need strong compactness of the local kinetic energy in order to pass to the limit $\e\to 0$. Spatial tightness, velocity averaging, and a local third-moment estimate based on the geometry of free trajectories in $\R_x^3$ provide the required compactness. The aforementioned Proposition \ref{prop:consist} allows us to correctly identify the limit of the regularized Fermi--Dirac distribution even at saturated limit points.

The exact conservation laws, as well as the dissipation of entropy, are recovered at the level of the limiting equation. The approximation preserves the Pauli bound and the uniform moment estimates, and its mass defect tends to zero. After passage to the limit, the moment matching in \eqref{eq/ emm} yields mass, momentum, and energy conservation directly from the mild formulation. The entropy inequality requires a further argument since $\eta_\hbar'$ has logarithmic singularities at both endpoints of the Pauli interval. Hence, we use the entropy-gap dissipation and approximate $\eta_\hbar'$ by bounded monotone functions, preserving convexity of the entropy approximation. This yields the H-theorem on the full Pauli interval, including saturated zero-temperature states.

For functions on phase space, we use
\[
 \|g\|_{L^1_2}  :=  \iint_{\R^3\times\R^3}  (1+|p|^2)|g(x,p)|\,\dx\dpp,
\]
and, for functions of momentum alone,
\[
 \|g\|_{L^1_2(\R^3_p)}  :=  \int_{\R^3}(1+|p|^2)|g(p)|\,\dpp.
\]
We also define
\bq \label{eq:H-def}
 \mathscr H_\hbar(g)  :=  \iint_{\R^3\times\R^3}  \eta_\hbar(g)\,\dx\dpp
\eq
and
\bq \label{eq:D-def}
 \mathscr D_\hbar(g)  :=  \int_{\R^3}  \lt(  \int_{\R^3}\eta_\hbar(g)\,\dpp  -  \int_{\R^3}\eta_\hbar(\calF(g))\,\dpp  \rt)\dx.
\eq
By the variational characterization of the local equilibrium, $\mathscr D_\hbar(g)\geq0$ (see Proposition \ref{prop:vari}).

With this notation, our first main theorem gives the global construction of a mild solution together with the conservation laws and the entropy inequality.

\begin{theorem}\label{thm:main}
Assume
\bq \label{eq:ini}
 0\leq f_0\leq\frac1\hbar \quad\text{a.e.},  \quad  \|f_0\|_{L^1_2}<\infty.
\eq
Then there exists
\[
 f\in C([0,\infty);L^1(\R^3\times\R^3)) \cap L^\infty([0,\infty);L^1_2(\R^3\times\R^3))
\]
such that $0\leq f\leq\frac1\hbar $ almost everywhere and
\bq \label{eq:mild-main}
 f(t,x,p) =e^{-t}f_0(x-tp,p) +\int_0^t e^{-(t-s)} \calF(f)(s,x-(t-s)p,p)\,\ds
\eq
for every $t\geq0$ in $L^1(\R^3\times\R^3)$.  Moreover,
\[
 \iint_{\R^3\times\R^3}  (1,p,|p|^2)f(t,x,p)\,\dx\dpp  =  \iint_{\R^3\times\R^3}  (1,p,|p|^2)f_0(x,p)\,\dx\dpp
\]
for every $t\geq0$.

If, in addition,
\[
 \iint_{\R^3\times\R^3}  |x|^2f_0(x,p)\,\dx\dpp<\infty,
\]
then $\mathscr H_\hbar(f(t))$ is finite for every $t\geq0$, $\mathscr D_\hbar(f)\in L^1_{\rm loc}([0,\infty))$, and
\bq\label{eq:H-thm-gap}
 \mathscr H_\hbar(f(t))  +  \int_0^t\mathscr D_\hbar(f(s))\,\ds  \leq  \mathscr H_\hbar(f_0)  \quad\text{for every }t\geq0.
\eq
\end{theorem}

\begin{remark}
The entropy-gap formulation is convenient since it remains finite and well defined on the entire Pauli interval, including the endpoint values $f=0$ and $f=\frac1\hbar $.  We note, however, that a stronger logarithmic entropy production can at least be formally obtained, if we ignore all subtleties coming from these singular endpoints. Indeed, writing $F=\calF(f)$, define formally
\[
 \mathscr D_\hbar^{\log}(f)  :=  \iint_{\R^3 \times \R^3}  (f-F)  \lt(  \log\frac{f}{1-\hbar f}  -  \log\frac{F}{1-\hbar F}  \rt) \dx\dpp .
\]
Then equivalently,
\[
 \mathscr D_\hbar^{\log}(f) = \iint_{\R^3 \times \R^3}  (f-F)  \log\lt(  \frac{f(1-\hbar F)}{F(1-\hbar f)}  \rt) \dx\dpp
 \ge0.
\]
When $F$ is not at zero temperature, we note $\eta_\hbar'(F)  =  -a|p-u|^2-c$. Thus, local moment matching \eqref{eq/ emm} gives
\[
 \int_{\R^3} \eta_\hbar'(F)(f-F)\,\dpp=0.
\]
Therefore, the chain rule (along characteristics) yields formally the stronger entropy identity
\[
\mathscr H_\hbar(f(t))  +  \int_0^t\mathscr D_\hbar^{\log}(f(s))\,\ds  =  \mathscr H_\hbar(f_0).
\]
Moreover, by convexity,
\[
 \mathscr D_\hbar^{\log}(f)\ge  \mathscr D_\hbar(f).
\]
If $f$ reaches $0$ or $\frac1\hbar $, however, the logarithmic entropy production may fail to be integrable. Nevertheless, the entropy-gap dissipation $\mathscr D_\hbar$ remains meaningful and \eqref{eq:H-thm-gap} continues to hold. When $f$ is saturated, the sharp Pauli inequality (Lemma \ref{lem:sharp-pauli} below) forces $f=\calF(f)$ almost everywhere in momentum, hence the entropy gap vanishes automatically.
\end{remark}

\begin{remark}
Strict separation from the Pauli bound is propagated for all positive times, and so certain assumptions on the initial data can be imposed if one wishes to prevent saturation of the solution. More precisely, assume in addition that
\[
 f_0(x,p)<\frac1\hbar   \quad\text{for almost every }(x,p).
\]
Using the mild formula and $0\le\calF(f)\le\frac1\hbar$, we obtain that for every $t>0$
\[
 f(t,x,p) \le  e^{-t}f_0(x-tp,p)  +  \frac1\hbar \int_0^t e^{-(t-s)}\,ds   =  e^{-t}f_0(x-tp,p)  +(1-e^{-t})\frac1\hbar   <\frac1\hbar 
\]
for almost every $(x,p)$. Lemma \ref{lem:sharp-pauli} shows that this implies
\[
 E_f(t,x)-\frac{|P_f(t,x)|^2}{N_f(t,x)}  >  \kappa_\hbar N_f(t,x)^{\frac53}
\]
almost everywhere on $\{N_f>0\}$. In particular, Lemma \ref{lem:sharp-pauli} states that this means the solution does not saturate, and thus the local equilibrium $\calF(f)$ remains on the finite-temperature Fermi--Dirac branch for all positive times.

Similarly, if $f_0>0$ almost everywhere, then
\[
 f(t,x,p)\ge e^{-t}f_0(x-tp,p)>0.
\]
 Lemma \ref{lem:sharp-pauli} tells us that this also excludes the possibility of the local equilibrium $\calF(f)$ having zero-temperature. 
\end{remark}

Theorem \ref{thm:main} concerns the quantum equation for a fixed $\hbar>0$.  Our second main result addresses the complementary question of whether the corresponding quantum solutions recover the classical BGK equation as $\hbar\to0$.  To state the result, for a nonnegative distribution $g$ with finite second moments $N_g>0$, define
\[
 u_g:=\frac{P_g}{N_g},  \quad  T_g:=\frac{E_g-|P_g|^2/N_g}{3N_g},
\]
and the local Maxwellian
\[
 M(g)(t,x,p)  :=  \frac{N_g}{(2\pi T_g)^{\frac32}}  \exp \lt(-\frac{|p-u_g|^2}{2T_g}\rt).
\]
At points where $N_g=0$, we define simply $M(g)=0$.

The compactness mechanism required for the semiclassical limit is different from the fixed-$\hbar$ construction, since the Pauli bound used in the proof of Theorem \ref{thm:main} degenerates as $\hbar\to0$. We therefore utilize the uniform entropy bound instead, which gives equiintegrability. The latter allows us to combine bounded renormalizations with $L^1$ velocity averaging, and then remove the renormalization by a quantitative $L\log L$ estimate. Interpolating the averaging lemma with the dispersive local gain-of-third-moments estimate, we deduce strong compactness of the density, momentum, and energy. A second issue is the identification of the $\hbar$-dependent local equilibria.  Their parameters are determined implicitly by the Fermi--Dirac moment map, while the saturation threshold itself varies with $\hbar$.  At almost every non-vacuum point, we show that positivity of the classical internal energy eventually forces desaturation of the local Fermi--Dirac distribution. Then, by delicately identifying the limits of the coefficients of the local Fermi--Dirac distribution, we prove that the latter converges strongly in $L^1$ to a local Maxwellian. Thanks to the mild formula this is enough to prove that the family of solutions to the quantum BGK equation converge, modulo subsequences, in $C([0,T];L^1(\R^3\times\R^3))$ to a solution of the classical BGK equation.

Put concisely, our Theorem \ref{thm:semi} below shows that under uniform moment and entropy bounds, every sequence of the quantum mild solutions constructed as in Theorem \ref{thm:main} admits a subsequence converging strongly to a mild solution of the classical BGK equation.

\begin{theorem}\label{thm:semi}
Let $\hbar_n\downarrow0$, and assume
\bq \label{eq:semi-ini-conv}
 0\leq f_{0,n}\leq\hbar_n^{-1},  \quad  f_{0,n}\to f_0  \quad\text{strongly in }L^1(\R^3\times\R^3),
\eq
for some $f_0\geq0$.  Suppose also that
\bq \label{eq:semi-ini-b}
 \sup_n  \iint_{\R^3\times\R^3}  (1+|x|^2+|p|^2)f_{0,n}\,\dx\dpp<\infty,  \quad  \sup_n  \iint_{\R^3\times\R^3}  f_{0,n}|\log f_{0,n}|\,\dx\dpp<\infty.
\eq
Let $f_n$ be global mild solutions supplied by Theorem \ref{thm:main} with $\hbar=\hbar_n$ and initial data $f_{0,n}$.  Then, after extraction of a subsequence,
\bq \label{eq:semi-f-str}
 f_n\to f  \quad\text{strongly in }  C([0,T];L^1(\R^3\times\R^3))  \quad\text{for every }T>0,
\eq
where $f$ satisfies
\[
\pa_t f+p\cdot\nabla_xf=M(f)-f,  \quad  f|_{t=0}=f_0,
\]
in the mild sense, namely
\bq \label{eq:classical-mild}
 f(t,x,p) =e^{-t}f_0(x-tp,p) +\int_0^t e^{-(t-s)}  M(f)(s,x-(t-s)p,p)\,\ds.
\eq
\end{theorem}

The assumptions on the approximating initial data in Theorem \ref{thm:semi} are satisfied by canonical Pauli truncations of any classical initial datum with finite moments and entropy. Thus, every such initial datum can be associated to an approximating family of solutions to the quantum BGK equation, whose limit is covered by Theorem \ref{thm:semi}.

\begin{corollary}\label{cor:canon}
Assume $f_0\geq0$, and
\[
 \iint_{\R^3\times\R^3}  (1+|x|^2+|p|^2)f_0\,\dx\dpp<\infty,  \quad  \iint_{\R^3\times\R^3}  f_0|\log f_0|\,\dx\dpp<\infty.
\]
For $0<\hbar\leq1$, set
\[
 f_0^\hbar:=\min\lt\{f_0,\frac1\hbar\rt\}.
\]
Then, for every sequence $\hbar_n\downarrow0$, the initial data $f_{0,n}:=f_0^{\hbar_n}$ satisfy the assumptions of Theorem \ref{thm:semi}.  Consequently, the corresponding quantum mild solutions possess a subsequence converging strongly in $C([0,T];L^1(\R^3\times\R^3))$ for every $T>0$ to a mild solution of the classical BGK equation with initial datum $f_0$.
\end{corollary}

\begin{remark} 
As aforementioned in the introduction, the proof of Theorem \ref{thm:semi} is delicate because we do not know whether the local Fermi--Dirac distribution on the right-hand side of \eqref{eq:FD-BGK} has zero-temperature or not. Indeed, the saturated zero-temperature branch has a very different formal semiclassical limit. For fixed mass $N>0$ and velocity $u$, let us write
\[
 F_\hbar^{\rm sat}(p)  =  \frac1\hbar  \one_{\{|p-u|\leq R_\hbar\}},  \quad  R_\hbar  =  \lt(\frac{3\hbar N}{4\pi}\rt)^{\frac13}.
\]
As $\hbar\to0$, its support shrinks to $p=u$ while its height diverges, thus
\[
 F_\hbar^{\rm sat}  \rightharpoonup N\delta_u
\]
in the sense of measures.  Hence following the zero-temperature boundary leads to a monokinetic measure rather than an $L^1$ Maxwellian. This mass-concentration is excluded with the help of the entropy-assumptions of Theorem \ref{thm:semi}. To see this, note that
\[
 \int_{\R^3}  F_\hbar^{\rm sat}  |\log F_\hbar^{\rm sat}|\,\dpp  =  N|\log\hbar|  \to \infty,
\]
and thus we expect that the local equilibria should not follow the zero-temperature saturated profiles as $\hbar \to 0$.

\end{remark}

The rest of the paper is organized as follows. In Section \ref{sec:approx}, we discuss the moment structure of the Fermi--Dirac equilibria. We then construct the regularized local Fermi--Dirac distribution, prove its properties, and finally use these to construct mild solutions to the regularized problem. In Section \ref{sec:pass-eps}, we obtain compactness of the macroscopic moments and then pass to the limit, in particular that of the regularized Fermi--Dirac distribution, thereby completing the solution-construction part of Theorem \ref{thm:main}. Section \ref{sec:H-thm} establishes the exact conservation laws and the entropy inequality for the mild solutions, finishing the proof of Theorem \ref{thm:main}. Finally, Section \ref{sec:semiclas} proves the semiclassical limit stated in Theorem \ref{thm:semi} as well as Corollary \ref{cor:canon}.

%
%
%
%
%
%

%
%
%
%
%
%

\section{Regularization of the Fermi--Dirac equilibrium}\label{sec:approx}

In this section we construct the regularized collision operator used in the proof of Theorem \ref{thm:main} and establish the estimates required for the limit $\e\to0$.  We first describe the admissible moment region, then introduce the regularized macroscopic parameters and the exact equilibrium map, and finally solve the regularized equation. 

Since Theorem \ref{thm:main} is concerned with the existence theory for fixed $\hbar>0$, throughout this section we omit the dependence on $\hbar$ whenever there is no risk of confusion.

%
%
%
%
%
%

\subsection{Moment identities and the admissible region}

We begin with the general Fermi--Dirac profile
\[
 F(p):=\frac{1}{e^{a|p-b|^2+c}+\hbar},  \quad a>0,\quad b\in\R^3,\quad c\in\R,
\]
and we recall
\[
 \beta_0(c)=\int_{\R^3}\frac{1}{e^{|p|^2+c}+\hbar}\,\dpp,  \quad  \beta_2(c)=\int_{\R^3}\frac{|p|^2}{e^{|p|^2+c}+\hbar}\,\dpp.
\]
A direct change of variables gives
\bq \label{eq:FD-mom-com}
 N_F=a^{-\frac32}\beta_0(c),\quad  P_F=bN_F,\quad  E_F=a^{-\frac52}\beta_2(c)+|b|^2N_F.
\eq
Thus, if $a=\beta_0(c)^{\frac23}N^{-\frac23}$, then
\bq\label{eq:FD-mom}
 N_F=N,\quad  E_F=\beta(c)^{-\frac53}N^{\frac53}+N|b|^2.
\eq

In the following lemma, we determine bounds for the macroscopic moments, assuming that they are generated by a function which obeys the Pauli exclusion principle. Also, the case of equality, namely the \textit{saturated} case, is precisely identified.

\begin{lemma}\label{lem:sharp-pauli}
Let $N>0$ and $P\in\R^3$.  Set
\[
 u:=\frac PN,  \quad  R_N:=\lt(\frac{3\hbar N}{4\pi}\rt)^{\frac13},  \quad  g_\star(p):=\frac1\hbar\one_{\{|p-u|\leq R_N\}}.
\]
If $g$ satisfies
\[
 0\leq g\leq\frac1\hbar ,  \quad  (1+|p|^2)g\in L^1(\R^3), \quad   \int_{\R^3}g\,\dpp=N,  \quad  \int_{\R^3}pg\,\dpp=P,
\]
then
\bq \label{eq:sharp-pauli}
 E_g-\frac{|P|^2}{N}  \geq \kappah N^{\frac53},  \quad  \kappah:=\frac{3^{\frac53}}{5(4\pi)^{\frac23}}\hbar^{\frac23}.
\eq
Equality in \eqref{eq:sharp-pauli} (saturation) holds if and only if $g=g_\star$ almost everywhere.
\end{lemma}

\begin{proof}
Since $g$ and $g_\star$ have the same mass,
\[
 \int_{\R^3}(g-g_\star)\,\dpp=0.
\]
Consequently,
\[
\int_{\R^3}|p-u|^2(g-g_\star)\,\dpp =  \int_{\R^3}(|p-u|^2-R_N^2)(g-g_\star)\,\dpp.
\]
On $\{|p-u|<R_N\}$, both factors on the right are nonpositive, while on $\{|p-u|>R_N\}$ both are nonnegative.  Hence the right-hand side is nonnegative.  Moreover,
\[
 \int_{\R^3}|p-u|^2g\,\dpp=E_g-\frac{|P|^2}{N},
\]
whereas
\[
 \int_{\R^3}|p-u|^2g_\star\,\dpp  =\frac{4\pi}{\hbar}\int_0^{R_N}r^4\dd r  =\frac{4\pi}{5\hbar}R_N^5 =\frac{3^{\frac53}}{5(4\pi)^{\frac23}}\hbar^{\frac23}N^{\frac53}.
\]
This proves \eqref{eq:sharp-pauli}.  If equality holds, then
\[
 (|p-u|^2-R_N^2)(g-g_\star)=0
\]
almost everywhere. Since the sphere $\{|p-u|=R_N\}$ has zero Lebesgue measure, this forces $g=\frac1\hbar $ almost everywhere inside the ball and $g=0$ almost everywhere outside. Thus $g=g_\star$.  The converse is immediate.
\end{proof}

To recover the finite-temperature parameters from the macroscopic moments, we next record the range and monotonicity of the scalar map $\beta$.

\begin{lemma}\label{lem:beta}
For each $\hbar>0$, the function $\beta^{\hbar} = \beta_0^{\hbar}/(\beta_2^{\hbar})^{\frac35}$ is continuous and strictly decreasing on $\R$.  Moreover,
\[
 \lim_{c\to+\infty}\beta^{\hbar}(c)=0,  \quad  \lim_{c\to-\infty}\beta^{\hbar}(c)=\beta^{\hbar}(-\infty) = \frac{5^{\frac35}(4\pi)^{\frac25}}{3\hbar^{\frac25}} .
\]
Hence $\beta^{\hbar}:\R\to(0,\beta^{\hbar}(-\infty))$ is a continuous strictly decreasing bijection.
\end{lemma}

\begin{proof}
For $\hbar=1$, the endpoint values, strict monotonicity, and the resulting invertibility were established in \cite[Lemma 2.1, Proposition 2.1, and Theorem 2.2]{BY20}. For general $\hbar>0$, we observe
\[
 \beta^{\hbar}_{k}(c)  =\frac1\hbar \beta_{k}^1(c-\log\hbar),
\]
and thus
\[
 \beta^{\hbar}(c) =\hbar^{-\frac25}\beta^1(c-\log\hbar).
\]
The conclusion follows immediately from the $\hbar=1$ result, together with
\[
 \lim_{c\to-\infty}\beta^1(c) =\frac{5^{\frac35}(4\pi)^{\frac25}}{3}.
\]
\end{proof}
We remark that a direct computation gives the useful identity
\bq \label{eq:beta-kappa}
 \beta^{\hbar}(-\infty)^{-\frac53}=\kappah.
\eq

%
%
%
%
%
%

\subsection{Regularized macroscopic parameters}

We next introduce the regularized macroscopic quantities and the corresponding regularized Fermi--Dirac profile $\calF^\e(f)$.  The construction is designed  to meet two complementary requirements. For each fixed $\e>0$, the degeneracy parameter stays in a compact subset of $(0,\beta(-\infty))$ so that the resulting collision map is globally Lipschitz in $L^1_2$.  At the same time, the regularized equilibrium satisfies moment estimates that are uniform as $\e\to0$.  The regularized profile has slightly perturbed moments, and Proposition \ref{prop:consist} shows that these moment defects vanish as $\e\to0$. Throughout this subsection, $0<\e<\frac12$ and we will say that $f=f(p)$ is \emph{Pauli-admissible} if $0\le f \le \frac1\hbar$ and $(1+|p|^2)f\in L^1_p(\R^3)$. 

For a Pauli-admissible $f=f(p)$, let us define
\bq \label{eq:reg-N-u}
 N_f^\e:=\frac{N_f}{1+\e N_f},  \quad  u_f^\e:=\frac{P_f}{N_f+\e(1+|P_f|)}.
\eq
For brevity in the calculations below we write $N^\e=N_f^\e$, $u^\e=u_f^\e$, and $(N,P,E)=(N_f,P_f,E_f)$. Set
\[
 D^\e:=N+\e(N+E),
\]
and, whenever $N+E>0$,
\[
 S^\e:=E-\frac{|P|^2}{D^\e}+\e(N+E),  \quad  \widetilde B^\e:=\frac{N}{(S^\e)^{\frac35}}.
\]
At $(N,P,E)=(0,0,0)$, set $\widetilde B^\e=0$.  The advantage of this choice of $\widetilde B^\e$ is that its denominator now satisfies 
\bq \label{eq:S-bds}
 \e(N+E)\leq S^\e\leq(1+\e)(N+E).
\eq
Here, the lower bound comes from the fact that $|P|^2\leq NE$. 

The quantity $S^\e$ still degenerates at the vacuum state $(N,P,E)=(0,0,0)$. Thus, we regularize the degeneracy parameter near vacuum by interpolating toward its upper truncation level. More precisely, choose
$\chi\in C^\infty([0,\infty))$ satisfying
\[
 0\leq\chi\leq1,  \quad  \chi(z)=0\ \text{for }0\leq z\leq1,  \quad  \chi(z)=1\ \text{for }z\geq2,
\]
and write $\chi_\e(z)=\chi(z/\e)$. First truncate the range by
\bq \label{eq:Bhat}
 \overline{B^\e}  :=\max\lt\{  \e\beta(-\infty),  \min\{(1-\e)\beta(-\infty),\widetilde B^\e\}  \rt\},
\eq
and then set
\[
 \be(N,P,E)  :=\chi_\e(N+E)\overline{B^\e}  +(1-\chi_\e(N+E))(1-\e)\beta(-\infty).
\]
Thus
\bq \label{eq:Beps-range}
 \e\beta(-\infty)\leq\be\leq(1-\e)\beta(-\infty).
\eq
We then define
\begin{equation*}
 c^\e:=\beta^{-1}(\be).
\end{equation*}
The important feature of the range truncation is that it keeps the argument of $\beta^{-1}$ uniformly away from both endpoints of the interval $(0,\beta(-\infty))$.  Figure \ref{fig:Beps-cons} gives a schematic picture.  If
\[
 c_-^\e:=\beta^{-1}((1-\e)\beta(-\infty))  \quad \text{and} \quad   c_+^\e:=\beta^{-1}(\e\beta(-\infty)),
\]
then \eqref{eq:Bhat} forces $\overline{B^\e}$ into the shaded horizontal strip and hence confines $c^\e$ to the compact interval $[c_-^\e,c_+^\e]$.

\begin{figure}[H]

\centering
\begin{tikzpicture}[x=1.15cm,y=1.05cm,>=Latex]
  \fill[black!6] (-4.75,0.82) rectangle (4.35,3.18);

  \draw[->] (-4.95,0) -- (4.65,0) node[right] {$c$};

  \draw[thick,domain=-4.55:4.30,samples=100,smooth,variable=\x]
  plot
  ({\x},{4/(1+exp(0.6203*(\x+0.315)))});

\node[anchor=south west] at (0.65,1.25) {$y=\beta(c)$};

  \draw[densely dashed] (-4.75,4.0) -- (4.15,4.0);
  \draw[dashed] (-4.75,3.18) -- (4.15,3.18);
  \draw[dashed] (-4.75,0.82) -- (4.15,0.82);
  \node[anchor=east] at (-4.82,4.0) {$\beta(-\infty)$};
  \node[anchor=east] at (-4.82,3.18) {$(1-\e)\beta(-\infty)$};
  \node[anchor=east] at (-4.82,0.82) {$\e\beta(-\infty)$};
  \node[anchor=east] at (-4.82,0.06) {$0$};

  \draw[<->] (4.02,3.18) -- (4.02,4.0);
  \node[anchor=west] at (4.10,3.59) {$\e\beta(-\infty)$};
  \draw[<->] (4.02,0.02) -- (4.02,0.82);
  \node[anchor=west] at (4.10,0.40) {$\e\beta(-\infty)$};

  \draw[densely dashed] (-2.50,0) -- (-2.50,3.18);
  \draw[densely dashed] (1.87,0) -- (1.87,0.82);
  \fill (-2.50,3.18) circle (1.4pt);
  \fill (1.87,0.82) circle (1.4pt);
  \node[anchor=north] at (-2.50,-0.04) {$c_-^\e$};
  \node[anchor=north] at (1.87,-0.04) {$c_+^\e$};

  \draw[<->,thick] (-2.50,-0.55) -- (1.87,-0.55);
  \node at (0.06,-0.88) {$c^\e\in[c_-^\e,c_+^\e]$};
  \node[align=center] at (1.75,2.52)
    {$\e\beta(-\infty)\leq B^\e\leq(1-\e)\beta(-\infty)$};
\end{tikzpicture}
\caption{Schematic range regularization of the degeneracy parameter. Thanks to the truncation of the range, $B^\e$ now lies in a compact subinterval of $(0,\beta(-\infty))$, equivalently keeping $c^\e=\beta^{-1}(B^\e)$ in a compact interval.}
\label{fig:Beps-cons}

\end{figure}

Finally, when $N>0$, we set
\begin{equation}\label{eq/areg}
a^\e:=\beta_0(c^\e)^{\frac23}(N^\e)^{-\frac23}
\end{equation}
and
\bq \label{eq:Feps}
 \calF^\e(f)(p) :=\frac1{e^{a^\e|p-u^\e|^2+c^\e}+\hbar}.
\eq
If $N=0$, define $\calF^\e(f)=0$.  This last convention is consistent with the limit $N^\e\downarrow0$ in $L^1(\R^3_p)$.

The range truncation and the vacuum interpolation play complementary roles. The former controls the inverse moment map, while the latter removes the singularity at vacuum without losing uniform moment control. In particular, the regularization acts at the level of the macroscopic parameters, so that the approximate profile remains within the Fermi--Dirac family. We emphasize that due to the change in coefficients, exact moment matching such as \eqref{eq/ emm} does not hold for $\calF^\e$; we are only guaranteed the following identities due to \eqref{eq:FD-mom-com},
\bq \label{eq:Feps-mom}
 N_{\calF^\e(f)}=N^\e,\quad  P_{\calF^\e(f)}=N^\e u^\e,\quad  E_{\calF^\e(f)}=(\be)^{-\frac53}(N^\e)^{\frac53}  +N^\e|u^\e|^2.
\eq

\begin{remark} 
The interpolation near vacuum is introduced not merely to remove the degeneracy of the parametrization, but also to retain control of the equilibrium moments in terms of the input moments. For example, simply replacing $S^\e$ by $S^\e+\e$ and setting
\[
 B_{\rm add}^\e:=\frac{N}{(S^\e+\e)^{\frac35}}
\]
would give
\[
 (B_{\rm add}^\e)^{-\frac53}(N^\e)^{\frac53}
 =
 \frac{S^\e+\e}{(1+\e N)^{\frac53}}.
\]
For fixed $\e>0$, this quantity tends to $\e$, rather than to zero, as $(N,P,E)$ approaches vacuum through $N>0$. Hence this regularization alone does not yield an estimate of the form
\[
 E_{\calF^\e(f)}\le C_\hbar(N_f+E_f),
\]
which is needed uniformly in $\e$ for the compactness argument.

The interpolation in the definition of $B^\e$ avoids this nonvanishing energy contribution near vacuum while retaining the lower bound on $B^\e$ required for the uniform moment estimate below.
\end{remark}

\begin{lemma}\label{lem:unif-e}
There is a constant $C_\hbar>0$, independent of $\e\in(0,\frac12)$ and of $f$, such that every Pauli-admissible $f$ satisfies
\[
 \int_{\R^3}(1+|p|^2)\calF^\e(f)(p)\,\dpp  \leq C_\hbar(N_f+E_f).
\]
\end{lemma}

\begin{proof}
By Lemma \ref{lem:sharp-pauli}, whenever $N>0$,
\begin{equation}\label{eq/ easy}
 \frac{N}{(N+E)^{\frac35}}  \leq \frac{N}{(E-|P|^2/N)^{\frac35}}  \leq\beta(-\infty).
\end{equation}
The assertion is trivial if $N=0$.  From \eqref{eq:S-bds},
\[
 \widetilde B^\e  \geq(1+\e)^{-\frac35}\frac{N}{(N+E)^{\frac35}}.
\]
The upper truncation value satisfies, simply because of \eqref{eq/ easy},
\[
 (1-\e)\beta(-\infty) \geq\frac12\frac{N}{(N+E)^{\frac35}}.
\]
Together, we see that the range-truncated $\overline{B^\e}$, and then also the convex combination $\be$, satisfy
\bq \label{eq:B-lower}
 \be(N,P,E)  \geq C_0\frac{N}{(N+E)^{\frac35}}
\eq
with a universal $C_0>0$, independent of $\e\in(0,\frac12)$.

Using \eqref{eq:Feps-mom}, $N^\e\leq N$, and \eqref{eq:B-lower},
\[
 (\be)^{-\frac53}(N^\e)^{\frac53}  \leq C(N+E).
\]
Moreover, for $N>0$, $|u^\e|\leq\frac{|P|}{N}$, and thus
\begin{equation}\label{eq:E-lower}
 N^\e|u^\e|^2 \leq\frac{|P|^2}{N}\leq E.
\end{equation}
This together with the following identity
\[
 \int_{\R^3}(1+|p|^2)\calF^\e(f)(p)\,\dpp = N^\e + (\be)^{-\frac53}(N^\e)^{\frac53} +  N^\e|u^\e|^2
\]
gives the desired result.  The case $N=0$ is immediate from the definition. This completes the proof.
\end{proof}

For fixed $\e>0$, the regularization also restores the Lipschitz property required for the Picard construction.

\begin{lemma}\label{lem:fix-e-Lip}
For each fixed $\e\in(0,\frac12)$, there exists $C_{\e,\hbar}<\infty$ such that, whenever $f,g$ are Pauli-admissible,
\[
 \|\calF^\e(f)-\calF^\e(g)\|_{L^1_2(\R^3_p)} \leq C_{\e,\hbar}\|f-g\|_{L^1_2(\R^3_p)}.
\]
\end{lemma}

\begin{proof}
If $N_f=0$ or $N_g=0$, then the assertion becomes trivial due to \eqref{eq:Feps-mom},  $N^\e\leq N$, and \eqref{eq:E-lower}. Let $N_f$ and $N_g$ be non-zero and set 
\[
 (N,P,E):=(N_f,P_f,E_f),  \quad  (N',P',E'):=(N_g,P_g,E_g).
\]
Since
\[
 |N-N'|+|P-P'|+|E-E'| \leq C\|f-g\|_{L^1_2(\R^3_p)},
\]
it suffices to establish Lipschitz dependence of the regularized parameters on the moments and then of the Fermi--Dirac profile on these parameters.

From \eqref{eq:reg-N-u},
\[
 |N_f^\e-N_g^\e|  =\frac{|N-N'|}{(1+\e N)(1+\e N')}  \leq |N-N'|.
\]
For the regularized velocity, set
\[
 d:=N+\e(1+|P|),  \quad  d':=N'+\e(1+|P'|).
\]
Since $d,d'\geq\e$ and $|P'|/d'\leq\frac1\e$, we obtain
\begin{align*}
 |u_f^\e-u_g^\e|  &\leq  \frac{|P-P'|}{d}  +\frac{|P'|}{d'}\frac{|d-d'|}{d} \\
 &\leq  \frac1\e|P-P'|  +\e^{-2}\lt(|N-N'|+\e|P-P'|\rt) \\
 &\leq  \e^{-2}|N-N'|  +2\frac1\e|P-P'|.
\end{align*}

We next estimate the regularized degeneracy parameter $B^\e$. On the region $N+E\geq\e$, write
\[
 D^\e=(1+\e)N+\e E,  \quad  S^\e  =E-\frac{|P|^2}{D^\e}+\e(N+E).
\]
Thanks to Jensen's inequality $ |P|^2\leq NE$, and so
\[
 |P|\leq\sqrt{NE}\leq\frac{N+E}{2},
\]
while
\[
 D^\e\geq\e(N+E),  \quad  S^\e\geq\e(N+E)\geq\e^2.
\]
A direct differentiation gives
\[
 \pa_N S^\e  =  \frac{(1+\e)|P|^2}{(D^\e)^2}+\e,  \quad  \nabla_P S^\e  =  -\frac{2P}{D^\e}, \quad  \pa_E S^\e  =  1+\e+\frac{\e|P|^2}{(D^\e)^2}.
\]
Using these bounds, we find
\bq \label{eq:Seps-deri-bd}
 |\pa_N S^\e|  +|\nabla_P S^\e|  +|\pa_E S^\e|  \leq C_\e  \quad\text{on }\{N+E\geq\e\}.
\eq

Since $N\leq N+E\leq\frac1\e S^\e$, differentiating $\widetilde B^\e  =N(S^\e)^{-\frac35}$ yields
\[
 \pa_N\widetilde B^\e   =  (S^\e)^{-\frac35}  -\frac35N(S^\e)^{-8/5}\pa_NS^\e,\quad \nabla_P\widetilde B^\e =  -\frac35N(S^\e)^{-8/5}\nabla_PS^\e,\quad  \pa_E\widetilde B^\e =  -\frac35N(S^\e)^{-8/5}\pa_ES^\e.
\]
Hence \eqref{eq:Seps-deri-bd}, together with $S^\e\geq\e^2$, gives
\bq \label{eq:Btilde-deri-bd}
 |\pa_N\widetilde B^\e|  +|\nabla_P\widetilde B^\e|  +|\pa_E\widetilde B^\e|  \leq C_\e  \quad\text{on }\{N+E\geq\e\}.
\eq

The max--min maps used in the range truncation \eqref{eq:Bhat} are one-Lipschitz.  In the transition region $ \e\leq N+E\leq2\e$, we also have
\[
 |\chi_\e'|\leq\frac{C}{\e},
\]
while both $\overline{B^\e}$ and $(1-\e)\beta(-\infty)$ are bounded by constants depending only on $\e$ and $\hbar$.  On $N+E\leq\e$, the function $B^\e$ is constant.   Together with \eqref{eq:Btilde-deri-bd}, this yields a bound $C_{\e,\hbar}$ for the derivatives of $B^\e$ wherever they exist. Consequently, the mean-value theorem shows
\bq\label{eq:Beps-L-mom}
 |B^\e(N,P,E)-B^\e(N',P',E')|  \leq C_{\e,\hbar}  \lt(|N-N'|+|P-P'|+|E-E'|\rt).
\eq

By \eqref{eq:Beps-range},
\[
 c^\e\in I_\e  :=  \lt[  \beta^{-1}((1-\e)\beta(-\infty)),  \beta^{-1}(\e\beta(-\infty))  \rt].
\]
Since $\beta'$ is continuous and strictly negative,
\[
 \inf_{c\in I_\e}|\beta'(c)|>0,
\]
and hence the inverse function theorem gives
\[
 |\beta^{-1}(z)-\beta^{-1}(z')|  \leq C_{\e,\hbar}|z-z'| \quad \text{for } z,z'\in[\e\beta(-\infty),(1-\e)\beta(-\infty)].
\]
Combining this with \eqref{eq:Beps-L-mom}, we obtain
\[
 |c_f^\e-c_g^\e|  \leq C_{\e,\hbar}  \lt(  |N-N'|+|P-P'|+|E-E'|  \rt).
\]

It remains to estimate the dependence of the Fermi--Dirac profile on these regularized parameters.  For $N>0$, $c\in I_\e$, and $u\in\R^3$, define
\[
 \Phi(N,c,u)(p)  :=  \frac1{  \exp\lt(\beta_0(c)^{\frac23}N^{-\frac23}|p-u|^2+c\rt)  +\hbar}.
\]
Due to the definition of $\calF^\e$, we only need consider parameters satisfying
\[
 0<N\leq\frac1\e,  \quad  |u|\leq\frac1\e.
\]
Set
\[
 Q  :=  \beta_0(c)^{\frac23}N^{-\frac23}|p-u|^2+c,  \quad  \Lambda(Q)  :=  \frac{e^Q}{(e^Q+\hbar)^2}.
\]
A direct differentiation gives
\begin{equation}\label{eq/ dir diff}
\begin{aligned}
 \pa_N\Phi  &=  \frac23\Lambda(Q)  \beta_0(c)^{\frac23}N^{-\frac53}|p-u|^2,\quad  \nabla_u\Phi  =  2\Lambda(Q)  (\beta_0(c))^{\frac23}N^{-\frac23}(p-u),\\
 \pa_c\Phi  &=  -\Lambda(Q)  \lt(  1+  \frac23\frac{\beta_0'(c)}{\beta_0(c)}  (\beta_0(c))^{\frac23}N^{-\frac23}|p-u|^2  \rt).
\end{aligned}
\end{equation}

We now claim that
\bq \label{eq:Phi-deri}
 \|\pa_N\Phi\|_{L^1_2(\R^3_p)}  +\|\nabla_u\Phi\|_{L^1_2(\R^3_p)}  +\|\pa_c\Phi\|_{L^1_2(\R^3_p)}  \leq C_{\e,\hbar}.
\eq
Indeed, making the change of variables
\[
 p  =  u+N^{\frac13}\beta_0(c)^{-\frac13}v,  \quad  \dpp  =  N\beta_0(c)^{-1}\,\dd v,
\]
we have
\[
 Q=|v|^2+c \quad \text{and} \quad N^{-\frac53} |p-u|^2 \,\dpp = \beta_0(c)^{-\frac{5}{3}} |v|^2 \, \mathrm{d}v.
\]
Thus, after the above change of variables, the $N$-factors in the integrals of \eqref{eq/ dir diff} are respectively of order $1$, $N^{\frac23}$, and $N$. Moreover, since $c\in I_\e$, both $\beta_0(c)$ and $\beta_0'(c)/\beta_0(c)$ are uniformly bounded on $I_\e$, with $\beta_0(c)$ bounded away from zero; further, $|u|\leq\frac1\e$, and $\Lambda(|v|^2+c)$ has Gaussian decay uniformly for $c\in I_\e$. This proves \eqref{eq:Phi-deri}. In particular, the estimate for $\pa_N\Phi$ remains uniform as $N\downarrow0$, while the other two terms vanish in that limit; so $\Phi$ extends continuously to $N=0$ by setting
\[
 \Phi(0,c,u):=0.
\]

Integrating the derivative bounds in \eqref{eq:Phi-deri} successively along the $N$-, $c$-, and $u$-segments, we obtain
\[
 \|\Phi(N,c,u)-\Phi(N',c',u')\|_{L^1_2(\R^3_p)}  \leq  C_{\e,\hbar}  \lt(  |N-N'|+|c-c'|+|u-u'|  \rt).
\]
Since
\[
 \calF^\e(f)  =  \Phi(N_f^\e,c_f^\e,u_f^\e),  \quad  \calF^\e(g)  =  \Phi(N_g^\e,c_g^\e,u_g^\e),
\]
combining all of the above estimates yields the desired result.
\end{proof}

%
%
%
%
%
%

\subsection{The equilibrium map on the admissible moment set}\label{sec:ext-equil}

 In this section, we investigate the consistency of the regularization $\calF^\e$ as $\e\downarrow 0$. Toward this end, we need to first show that it is possible to match, to each Pauli-admissible $f$, a unique local equilibrium $\calF(f)$ which is determined by its moments. As we will see, the ordinary Fermi-Dirac distribution will be its interior representation, whereas its saturated form will be the boundary representation.

We need some notations. For $N>0$, set
\[
 E_\star(N,P):=\frac{|P|^2}{N}+\kappah N^{\frac53}
\]
and define the non-vacuum admissible moment set
\[
 \mathfrak M_\hbar :=\{(N,P,E) \in \R \times \R^3 \times \R :N>0,\ E \ge E_\star(N,P)\}.
\]
For a non-negative function $f:\R^3\to \R$, we may naturally generate a moment triple $(N_f, P_f, E_f)$. By Lemma \ref{lem:sharp-pauli}, if $f$ satisfies the Pauli bound and has finite second moments, then its moment triple belongs to $\mathfrak M_\hbar$ (assuming $N_f > 0$). As a type of converse statement, Proposition \ref{prop:vari} below shows that for every triple in $\mathfrak M_\hbar$, we may construct a Pauli-admissible distribution of the form \eqref{eq:F-s}. The vacuum triple is $(0,0,0)$; it will always be handled separately by setting the local equilibrium equal to zero.

For $(N,P,E)\in\mathfrak M_\hbar$, let
\[
 \calC_\hbar(N,P,E)  :=\lt\{  g: 0\leq g\leq\frac1\hbar ,\ (1+|p|^2)g\in L^1(\R^3),\  \int_{\R^3}(1,p,|p|^2)g\,\dpp=(N,P,E)  \rt\}.
\]
We also recall the Fermi-Dirac entropy density
\begin{equation} \label{eq/ entropy dens}
 \eta_\hbar(z) =z\log z+\frac1\hbar(1-\hbar z)\log(1-\hbar z),  \quad 0\leq z\leq\frac1\hbar ,
\end{equation}
with the convention $0\log0=0$. Now using the preceding moment characterization, let us define the exact equilibrium (including the zero-temperature case) as the unique minimizer of the Fermi-Dirac entropy.

\begin{proposition}\label{prop:vari}
For every $(N,P,E)\in\mathfrak M_\hbar$, the problem
\[
 \inf_{g\in\calC_\hbar(N,P,E)}  \int_{\R^3}\eta_\hbar(g(p))\,\dpp
\]
has a unique minimizer, denoted by $\calF[N,P,E]$.  If $E>E_\star(N,P)$, then
\bq \label{eq:inte-FD}
 \calF[N,P,E](p)  =\frac1{e^{a|p-u|^2+c}+\hbar},  \quad  u=\frac PN,
\eq
where $c\in\R$ and $a>0$ are determined by
\[
 \beta(c)=B(N,P,E),  \quad  a=\beta_0(c)^{\frac23}N^{-\frac23}.
\]
If $E=E_\star(N,P)$, then
\bq \label{eq:satu-FD}
 \calF[N,P,E](p)  =\frac1\hbar\one_{\{|p-P/N|\leq R_N\}},  \quad  R_N=\lt(\frac{3\hbar N}{4\pi}\rt)^{\frac13}.
\eq
In both cases,
\[
 \int_{\R^3}(1,p,|p|^2)\calF[N,P,E](p)\,\dpp=(N,P,E).
\]
\end{proposition}

\begin{proof}

Let us first point out that $\eta_{\hbar}(g)$ is indeed integrable for each $g\in \calC_{\hbar}(N,P,E)$, so all of the integrals written in the proof below are well-defined. Indeed, each term of $\eta_{\hbar}(z)$ may be estimated as such: for the $z\log z$ term in \eqref{eq/ entropy dens}, writing $V(p)=1+|p|^2$, we have
\[
 z\log z=z\log(ze^V)-Vz\geq-e^{-V-1}-Vz,
\]
and $(z\log z)^+\leq C_\hbar z$ on $[0,\frac1\hbar ]$. For the second term in \eqref{eq/ entropy dens}, we may use the elementary relation
\[
 -z\leq\frac1\hbar(1-\hbar z)\log(1-\hbar z)\leq0.
\] 
These relations together imply that the moment bounds provided by $g\in \calC_{\hbar}(N,P,E)$ imply the integrability of $\eta_{\hbar}(g)$.

By Lemma \ref{lem:sharp-pauli} and \eqref{eq:beta-kappa}, we know that
\[
 E\geq E_\star(N,P)  \quad\Longleftrightarrow\quad  B(N,P,E)\leq\beta(-\infty).
\]
We therefore split into two cases. Assume first that $E>E_\star(N,P)$. Lemma \ref{lem:beta} gives a unique $c\in\R$ satisfying $\beta(c)=B(N,P,E)$.  Let $F$ denote the function in \eqref{eq:inte-FD}.  With
\[
 p=u+N^{\frac13}\beta_0(c)^{-\frac13}v,
\]
the moment identities \eqref{eq:FD-mom-com}--\eqref{eq:FD-mom} give
\[
 \int_{\R^3}F\,\dpp=N,  \quad  \int_{\R^3}pF\,\dpp=P,
\]
and
\[
 \int_{\R^3}|p|^2F\,\dpp  =\beta(c)^{-\frac53}N^{\frac53}+N|u|^2=E.
\]
Thus $F\in\calC_\hbar(N,P,E)$.

For $z\in(0,\frac1\hbar )$,
\[
 \eta_\hbar'(z)=\log\frac{z}{1-\hbar z},  \quad  \eta_\hbar''(z)=\frac1{z(1-\hbar z)}>0.
\]
Consequently, an easy calculation shows
\[
 \eta_\hbar'(F(p))=-a|p-u|^2-c,
\]
which is a linear combination of $1,p_1,p_2,p_3,|p|^2$. Thus, the moment constraints along with the convexity of $\eta_{\hbar}$ show that for every $g\in\calC_\hbar(N,P,E)$:
\[
 \int_{\R^3}\lt(\eta_\hbar(g)-\eta_\hbar(F)\rt)\,\dpp \geq\int_{\R^3}\eta_\hbar'(F)(g-F)\,\dpp =0,
\]
and thus $F$ indeed a minimizer of this problem. Strict convexity yields uniqueness of $F$.

It only remains to discuss the remaining case of $E=E_\star(N,P)$. In this case we know from Lemma \ref{lem:sharp-pauli} that $\calC_\hbar(N,P,E)$ consists of the single function in \eqref{eq:satu-FD}. Thus the proof is complete.
\end{proof}

For a Pauli-admissible function $f=f(p)$ with finite local second moments, define
\[
 \calF(f)(p)  :=\calF[N_f,P_f,E_f]  \quad\text{when }N_f>0,
\]
and set $\calF(f)=0$ when $N_f=0$. The latter definition is natural because nonnegativity and zero mass imply $f\equiv 0$ almost everywhere. In the case where $f$ is dependent also on time and space, we define $\calF(f)(t,x,p)$ pointwisely for each $(t,x)$ in the obvious way.

To analyze the equilibrium map near the saturation boundary, we need the asymptotic behavior of the finite-temperature parameters as $c\to-\infty$. The following asymptotics help identify the limiting Fermi radius.

\begin{lemma}\label{lem:beta0-ze-t}
As $C\to+\infty$,
\[
 \beta_0(-C)=\frac{4\pi}{3\hbar}C^{\frac32}(1+o(1)), \quad \beta_2(-C)=\frac{4\pi}{5\hbar}C^{\frac52}(1+o(1)).
\]
\end{lemma}

\begin{proof}
Making the change of variables $p=\sqrt C\,q$, there holds
\[
 C^{-\frac32}\beta_0(-C)  =4\pi\int_0^\infty\frac{r^2}{e^{C(r^2-1)}+\hbar}\,\dr.
\]
On $r\in(0,1)$ the integrand is bounded by $\frac1\hbar r^2$ and increases pointwisely to it.  On $(1,\infty)$, for $C\geq1$,
\[
 \frac{r^2}{e^{C(r^2-1)}+\hbar}  \leq r^2e^{-(r^2-1)},
\]
the right-hand side of which is integrable.  Using the dominated convergence theorem, we obtain the asymptotic formula for $\beta_0(-C)$.  The same argument, with $r^4$ in place of $r^2$, gives the assertion for $\beta_2(-C)$.
\end{proof}

We now state for later use an elementary form of Scheff\'e's lemma, which will be used repeatedly.

\begin{lemma} \label{lem:pos-L1}
Let $(X,\mu)$ be a measure space and let $h_j,h\geq0$ belong to $L^1(X,\mu)$.  If $h_j\to h$ almost everywhere and
\[
 \int_X h_j\,\dd\mu \to \int_X h\,\dd\mu,
\]
then $h_j\to h$ strongly in $L^1(X,\mu)$.
\end{lemma}

\begin{proof}
Since $\min\{h_j,h\}\to h$ almost everywhere and $0\leq\min\{h_j,h\}\leq h$, dominated convergence gives
\[
 \int_X \min\{h_j,h\} \, \dd\mu \to \int_X h\,\dd\mu.
\]
The identity
\[
 \int_X |h_j-h|\,\dd\mu  =\int_X h_j\,\dd\mu+\int_X h\,\dd\mu  -2\int_X \min\{h_j,h\}\,\dd\mu
\]
proves the claim.
\end{proof}

The following proposition is unrelated to the regularization, but of independent interest. It shows that we can prove continuity of the exact equilibrium map, on the entire admissible moment set. The nonstandard point is continuity at the saturation boundary, where the finite-temperature parameter $c$ tends to $-\infty$. Lemma \ref{lem:beta0-ze-t} identifies the pointwise limit, while Lemma \ref{lem:pos-L1} upgrades it to strong weighted $L^1$ convergence.

\begin{proposition}
Let $(N_j,P_j,E_j),(N,P,E)\in\mathfrak M_\hbar$ and assume
\begin{equation}\label{eq/ mom conv}
 (N_j,P_j,E_j)\to(N,P,E).
\end{equation}
Then
\bq \label{eq:conti}
 \calF[N_j,P_j,E_j]  \to \calF[N,P,E]  \quad\text{strongly in }L^1_2(\R^3_p).
\eq
\end{proposition}

\begin{proof}
Set $F_j=\calF[N_j,P_j,E_j]$ and $F=\calF[N,P,E]$. Proposition \ref{prop:vari} and the assumption \eqref{eq/ mom conv} already tell us that
\[
 \int_{\R^3}(1+|p|^2)F_j\,\dpp=N_j+E_j  \to N+E  =\int_{\R^3}(1+|p|^2)F\,\dpp.
\]
In view of Lemma \ref{lem:pos-L1} it therefore suffices to check that $F_j(p)\to F(p)$ a.e., and \eqref{eq:conti} would be proven.

Assume that $E>E_\star(N,P)$. Then eventually $E_j>E_\star(N_j,P_j)$, and continuity of $\beta^{-1}$ on $(0,\beta(-\infty))$ gives
\[
 c_j\to c,  \quad  u_j=\frac{P_j}{N_j}\to\frac PN,  \quad  a_j\to a.
\]
Thus $F_j(p)\to F(p)$ for every $p$ and we obtain \eqref{eq:conti}.

Assume now that $E=E_\star(N,P)$.  For indices satisfying $E_j=E_\star(N_j,P_j)$, the profiles are saturated distributions: $F_j(p) = \frac1\hbar\mathbf{1}_{|p-P_j/N_j|\le R_{N_j}}$. Since $P_j/N_j\to P/N$ and $R_{N_j}\to R_N$, clearly $F_j(p)$ is a.e. convergent to $F(p)$ and we are done. Next, it remains to prove convergence for the indices where $E_j > E_\star(N_j,P_j)$. By assumption,
\[
 B(N_j,P_j,E_j)\to\beta(-\infty),
\]
and thus $c_j\to-\infty$.  Set  
\[
C_j=-c_j \quad \text{and} \quad r_j^2:=\frac{C_j}{a_j}.
\]
By Lemma \ref{lem:beta0-ze-t},
\[
 r_j^3  =\frac{C_j^{\frac32}N_j}{\beta_0(-C_j)}  \to\frac{3\hbar N}{4\pi}=R_N^3.
\]
Hence $r_j\to R_N$, $u_j\to u=P/N$, and $a_j=C_j/r_j^2\to+\infty$.  Since
\[
 a_j|p-u_j|^2+c_j  =a_j\lt(|p-u_j|^2-r_j^2\rt),
\]
we obtain for every $p$ with $|p-u|\neq R_N$,
\[
 F_j(p)\to\frac1\hbar\one_{\{|p-u|<R_N\}}=F(p)  \quad\text{a.e. in }p.
\]
This completes the proof.
\end{proof}

In particular, if a sequence of admissible moment triples converges to the vacuum triple $(0,0,0)$, the corresponding exact equilibria converge to zero immediately, since
\[
 \|\calF[N_j,P_j,E_j]\|_{L^1_2(\R^3_p)}=N_j+E_j\to0.
\]

The next proposition is the one that is of practical use to us, as it gives the consistency of the $\e$-regularization with respect to the exact equilibrium. Note carefully that the argument of the regularized Fermi-Dirac distribution is also assumed to change with the parameter $\e$. Indeed, this is the form needed later in the compactness argument. Convergence of the macroscopic moments determines the strong limit of the regularized equilibrium profile throughout the admissible moment set, including the saturation boundary.

\begin{proposition}\label{prop:consist}
Let $\e_j\downarrow0$, and let $(N_j,P_j,E_j)$ be moment triples generated by Pauli-admissible functions $\{f_j\}$. We assume
\[
 (N_j,P_j,E_j)\to(N,P,E),
\]
where either $(N,P,E)\in\mathfrak M_\hbar$ or $(N,P,E)=(0,0,0)$. Let $F_j^{\rm reg} := \calF^{\e_j}(f_j)$ be the regularized profile obtained by \eqref{eq:reg-N-u}--\eqref{eq:Feps}, with $\e=\e_j$. Then
\bq \label{eq:consi}
 F_j^{\rm reg}\to\calF[N,P,E] \quad\text{strongly in }L^1_2(\R^3_p),
\eq
where $\calF[0,0,0]:=0$.
\end{proposition}

\begin{proof}
Throughout this proof let us denote by $N_j^{\e_j},u_j^{\e_j},c^{\e_j}_j,\ldots$ the quantities defined through \eqref{eq:reg-N-u}--\eqref{eq/areg}, by using $(N_j,P_j,E_j)$ in place of $(N,P,E)$.

If $(N,P,E)=(0,0,0)$, Lemma \ref{lem:unif-e} gives
\[
 \|F_j^{\rm reg}\|_{L^1_2(\R^3_p)}  \leq C_\hbar(N_j+E_j)\to0.
\]
Assume from now on that $N>0$. Since $N_j\to N$, eventually $N_j\geq N/2$, and so
\bq \label{eq:consi-Nu}
 N_j^{\e_j}=\frac{N_j}{1+\e_jN_j}\to N,  \quad  u_j^{\e_j}=\frac{P_j}{N_j+\e_j(1+|P_j|)}\to\frac PN.
\eq
Moreover, $\chi_{\e_j}(N_j+E_j)=1$ for all sufficiently large $j$, and
\[
 S_j^{\e_j} =E_j-\frac{|P_j|^2}{N_j+\e_j(N_j+E_j)}    +\e_j(N_j+E_j) \to E-\frac{|P|^2}{N}.
\]
Hence $\widetilde B_j^{\e_j}\to B(N,P,E)$.  Since the lower and upper truncation levels tend to $0$ and $\beta(-\infty)$, respectively,
\bq \label{eq:consi-B}
 B_j^{\e_j}\to B(N,P,E).
\eq

Let us consider two cases separately. If $B(N,P,E)<\beta(-\infty)$, then by continuity of $\beta^{-1}$ on $(0,\beta(-\infty))$, we deduce from \eqref{eq:consi-B} that $c^{\e_j}_j \to c = \beta^{-1}(B(N,P,E))$. Combining this with \eqref{eq:consi-Nu}, we obtain 
\[
 F_j^{\rm reg}(p)\to\calF[N,P,E](p).
\]
Furthermore, from \eqref{eq:Feps-mom} and \eqref{eq:consi-Nu}:
\[
 \int_{\R^3}(1+|p|^2)F_j^{\rm reg}\,\dpp =N_j^{\e_j}  +(B_j^{\e_j})^{-\frac53}(N_j^{\e_j})^{\frac53}  +N_j^{\e_j}|u_j^{\e_j}|^2 \to N+E.
\]
Applying Lemma \ref{lem:pos-L1} with $\mathrm{d}\mu=(1+|p|^2)\dpp$ proves \eqref{eq:consi}.

Suppose finally that $B(N,P,E)=\beta(-\infty)$. Then \eqref{eq:consi-B} shows $c_j^{\e_j}=\beta^{-1}(B_j^{\e_j})\to-\infty$.  Set  
\[
C_j=-c_j^{\e_j} \quad \text{and} \quad  r_j^2:=\frac{C_j}{a_j^{\e_j}}.
\]
Since $a_j^{\e_j}  =\beta_0(-C_j)^{\frac23}(N_j^{\e_j})^{-\frac23}$, Lemma \ref{lem:beta0-ze-t} and \eqref{eq:consi-Nu} give
\[
 r_j^3  =\frac{C_j^{\frac32}N_j^{\e_j}}{\beta_0(-C_j)}  \to\frac{3\hbar N}{4\pi}=R_N^3.
\]
Thus the profiles converge pointwise away from $|p-P/N|=R_N$ to the saturated equilibrium \eqref{eq:satu-FD}.  Their weighted integrals converge to
\[
 N+\beta(-\infty)^{-\frac53}N^{\frac53}+N\lt|\frac PN\rt|^2=N+E,
\]
thanks to the preceding moment formula and by recalling the assumption $B(N,P,E)=\beta(-\infty)$. A final application of Lemma \ref{lem:pos-L1} proves \eqref{eq:consi}.
\end{proof}

%
%
%
%
%
%

\subsection{The regularized equation}

For fixed $\e\in(0,\frac12)$, consider
\bq \label{eq:reg-PDE}
\pa_t f^\e+p\cdot\nabla_xf^\e  =\calF^\e(f^\e)-f^\e,  \quad  f^\e|_{t=0}=f_0.
\eq

The preceding growth and Lipschitz estimates allow us to solve the regularized equation by Picard iteration on arbitrary finite time intervals.

\begin{proposition}\label{prop:regul-exi}
Under \eqref{eq:ini}, for every fixed $\e\in(0,\frac12)$ there exists
\[
 f^\e\in C([0,\infty);L^1_2(\R^3 \times \R^3))
\]
satisfying
\bq \label{eq:reg-mild}
 f^\e(t,x,p)  =e^{-t}f_0(x-tp,p)  +\int_0^t e^{-(t-s)}  \calF^\e(f^\e)(s,x-(t-s)p,p)\,\ds.
\eq
Moreover,
\bq \label{eq:reg-pauli}
 0\leq f^\e\leq\frac1\hbar ,
\eq
and, for every $T>0$,
\bq \label{eq:unif-f-eps}
 \sup_{0<\e<\frac12}\sup_{0\leq t\leq T}  \|f^\e(t)\|_{L^1_2}\leq C_{T,\hbar,f_0}.
\eq
The same bound holds for $\calF^\e(f^\e)$ in place of $f^\e$.
\end{proposition}

\begin{proof}
Fix $\e\in(0,\frac12)$ and $T>0$. We construct the solution directly on $[0,T]$ by Picard iteration. Set
\[
 f^{0,\e}(t,x,p):=e^{-t}f_0(x-tp,p),
\]
and define recursively
\bq \label{eq:reg-Picard}
 f^{n+1,\e}(t,x,p) =e^{-t}f_0(x-tp,p) +\int_0^t e^{-(t-s)}  \calF^\e(f^{n,\e})(s,x-(t-s)p,p)\,\ds,  \quad n\geq0.
\eq
Since $0\leq\calF^\e(g)\leq\frac1\hbar $, induction gives
\[
 0\leq f^{n+1,\e}(t,x,p) \leq e^{-t}\frac1\hbar   +\frac1\hbar \int_0^t e^{-(t-s)}\,\ds =\frac1\hbar .
\]
The strong continuity of free transport in $L^1_2$, together with Lemmas \ref{lem:unif-e} and \ref{lem:fix-e-Lip}, shows inductively that every iterate belongs to $C([0,T];L^1_2)$. Moreover,
\[
 \sup_{0\leq t\leq T}\|f^{1,\e}(t)-f^{0,\e}(t)\|_{L^1_2} \leq C_\hbar T\|f_0\|_{L^1_2}.
\]
For $n\geq1$, applying Lemma \ref{lem:fix-e-Lip} to the difference between the iteration formulas for $f^{n+1,\e}$ and $f^{n,\e}$, and using the change of variables $y=x-(t-s)p$, which preserves the weight $1+|p|^2$, we obtain, for any $T>0$ and $t\in[0,T]$,
	\begin{align*}
		\|f^{n+1,\e}(t)-f^{n,\e}(t)\|_{L^1_2}&\leq C_{\e,\hbar}\int_0^t e^{-(t-s_1)}  \|f^{n,\e}(s_1)-f^{n-1,\e}(s_1)\|_{L^1_2}\,\ds_1 \\
		&\leq C_{\e,\hbar}^n\int_0^t\cdots \int_0^{s_{n-1}}   \|f^{1,\e}(s_n)-f^{0,\e}(s_n)\|_{L^1_2}\,\ds_n\cdots \ds_1 \\
		&\le C_{\hbar,f_0,T}\frac{\left(C_{\e,\hbar}t\right)^n}{n!}.
	\end{align*}
Consequently, for $n>m$,
\[
 \sup_{0\le t\le T}  \|f^{n,\e}(t)-f^{m,\e}(t)\|_{L^1_2}  \le C_{\hbar,f_0,T} \sum_{k=m}^{\infty}\frac{(C_{\e,\hbar}T)^k}{k!},
\]
which tends to zero as $m\to\infty$. Hence $\{f^{n,\e}\}_{n\ge0}$ is Cauchy in $C([0,T];L^1_2(\R^3\times\R^3))$.

Let $f^\e\in C([0,T];L^1_2(\R^3\times\R^3))$ denote its limit. The fixed-$\e$ Lipschitz estimate in Lemma \ref{lem:fix-e-Lip} implies
\[
 \calF^\e(f^{n,\e})\to\calF^\e(f^\e)  \quad\text{in }C([0,T];L^1_2(\R^3\times\R^3)),
\]
and therefore we may pass to the limit in \eqref{eq:reg-Picard} to obtain \eqref{eq:reg-mild}. The Pauli bound \eqref{eq:reg-pauli} follows from the corresponding bound for the iterates. Since $T>0$ is arbitrary and the limits agree on overlapping intervals, this defines $f^\e$ globally in time.

It remains to obtain an estimate independent of $\e$.  Lemma \ref{lem:unif-e} and \eqref{eq:reg-mild} give, for $0\leq t\leq T$,
\[
 \|f^\e(t)\|_{L^1_2}  \leq e^{-t}\|f_0\|_{L^1_2}  +C_\hbar\int_0^t e^{-(t-s)}\|f^\e(s)\|_{L^1_2}\,\ds.
\]
Multiplying by $e^t$ and applying Gr\"onwall's inequality yields \eqref{eq:unif-f-eps}.  Lemma \ref{lem:unif-e} gives the same estimate for $\calF^\e(f^\e)$. This completes the proof.
\end{proof}

%
%
%
%
%
%

\section{Compactness and passage to the limit}\label{sec:pass-eps}

Let $\{f^\e\}_{0<\e<\frac12}$ be the family of solutions to the approximating model constructed in Proposition \ref{prop:regul-exi}.  We first establish strong compactness of the density, momentum, and energy fields.  The continuity properties of the moment-to-equilibrium map from Section \ref{sec:ext-equil} then allow us to pass to the nonlinear relaxation term directly from the convergence of these three moments, including at points where the limiting moments lie on the saturation boundary.

%
%
%
%
%
%
\subsection{Compactness of the macroscopic moments}

Our goal in this subsection is to obtain strong compactness of
\[
 (N_{f^\e},P_{f^\e},E_{f^\e}).
\]
Velocity averaging directly gives compactness only for velocity averages with compactly supported test functions.  To reach the physical moments $1$, $p$, and $|p|^2$, we need two additional ingredients: spatial tightness of the density and uniform control of the velocity tails.  We first quantify the mass defect created by the regularization and derive spatial tightness.  We then obtain a local third-moment bound, which provides the additional velocity-tail control required for the energy.

We begin with the total mass. Set
\[
 M^\e(t):=\iint_{\R^3\times\R^3}f^\e(t,x,p)\,\dx\dpp.
\]
Since
\[
 N_{\calF^\e(f^\e)}=\frac{N_{f^\e}}{1+\e N_{f^\e}},
\]
we first derive a quantitative control of the small mass defect generated by the regularization. 
Let $\zeta_R\in C_c^\infty(\R^3)$ satisfy $0\leq\zeta_R\leq1$, $\zeta_R\uparrow1$, and $|\nabla\zeta_R|\leq C/R$.  Testing \eqref{eq:reg-PDE} against $\zeta_R(x)$ and then letting $R\to\infty$, the transport contribution vanishes since
\[
 \iint_{\R^3\times\R^3}|p|f^\e(t,x,p)\,\dx\dpp  \leq \|f^\e(t)\|_{L^1}^{\frac12}  \lt(\iint_{\R^3\times\R^3}|p|^2f^\e(t,x,p)\,\dx\dpp\rt)^{\frac12},
\]
uniformly on finite time intervals.  Hence
\[
 \frac{\dd}{\dd t}M^\e(t)  =-\int_{\R^3}\frac{\e N_{f^\e}(t,x)^2}  {1+\e N_{f^\e}(t,x)}\,\dx  \leq0.
\]
For $y\geq0$,
\[
 \frac{y}{1+y}\leq y^{\frac23}.
\]
Moreover, by Lemma \ref{lem:sharp-pauli}, we obtain
\[
 \frac{\e N_{f^\e}^2}{1+\e N_{f^\e}}  =N_{f^\e}\frac{\e N_{f^\e}}{1+\e N_{f^\e}}  \leq \e^{\frac23}N_{f^\e}^{\frac53}  \leq \kappah^{-1}\e^{\frac23}E_{f^\e},
\]
and Proposition \ref{prop:regul-exi} gives, for every $T>0$,
\bq \label{eq:m-loss-s}
 0\leq M^\e(0)-M^\e(t)  \leq C_{T,\hbar,f_0}\e^{\frac23},  \quad 0\leq t\leq T.
\eq

Integrating \eqref{eq:reg-PDE} in $p$ gives
\bq \label{eq:den-bal}
\pa_tN_{f^\e}+\nabla_x\cdot P_{f^\e}  =\frac{N_{f^\e}}{1+\e N_{f^\e}}-N_{f^\e}  \leq0.
\eq
Let $\chi_R$ be smooth, equal to zero on $B_R$, equal to one outside $B_{2R}$, and satisfy $|\nabla\chi_R|\leq C/R$.  Testing \eqref{eq:den-bal} by $\chi_R$ gives
\[
 \int_{\R^3}\chi_RN_{f^\e}(t,x)\,\dx  \leq\int_{\R^3}\chi_RN_{f_0}(x)\,\dx  +\frac{C}{R}\int_0^t\int_{\R^3}|P_{f^\e}(s,x)|\,\dx\ds.
\]
The first velocity moment is bounded by mass and kinetic energy, hence Proposition \ref{prop:regul-exi} implies
\bq \label{eq:spa-tig}
 \lim_{R\to\infty}  \sup_{0<\e<\frac12}\sup_{0\leq t\leq T}  \int_{|x|>R}N_{f^\e}(t,x)\,\dx=0.
\eq

The local gain of one velocity moment used below is classical in whole-space BGK compactness arguments; compare \cite{Per89}.  We record the short proof since it also indicates why the estimate is specific to $x\in\R^3$.

\begin{lemma}\label{lem:3-mom}
Let $T,R>0$, and suppose
\[
 f(t,x,p)  =e^{-t}f_0(x-tp,p) +\int_0^te^{-(t-s)}G(s,x-(t-s)p,p)\,\ds,
\]
where $f_0,G\geq0$.  Then
\begin{align*}
 &\int_0^T\int_{B_R}\int_{\R^3}|p|^3f(t,x,p)\,\dpp\dx\dt \\
 &\quad\leq 2R\lt(  \iint_{\R^3\times\R^3}|p|^2f_0(x,p)\,\dx\dpp  +\int_0^T\iint_{\R^3\times\R^3}|p|^2G(s,x,p)\,\dx\dpp\ds  \rt).
\end{align*}
\end{lemma}

\begin{proof}
We first record a geometric estimate for free trajectories. For $y,p\in\R^3$, set
\[
 A_{R,T}(y,p)  :=  \int_0^T \one_{B_R}(y+tp)\,\dt  =  \lt|\{t\in[0,T]:y+tp\in B_R\}\rt|.
\]
Thus $A_{R,T}(y,p)$ is the amount of time that the free trajectory $t\mapsto y+tp$ spends in $B_R$ during $[0,T]$.

Suppose first that $p\neq0$, and set
\[
 e:=\frac{p}{|p|},  \quad  y_\perp:=y-(y\cdot e)e.
\]
Then, we have $ y+tp  =  y_\perp+(y\cdot e+t|p|)e$, and consequently
\[
 |y+tp|^2  =  |y_\perp|^2+(y\cdot e+t|p|)^2
\]
due to $y_\perp\cdot e=0$. If $|y_\perp|\geq R$, then the line $\{y+tp:t\in\R\}$ does not enter $B_R$, up to the irrelevant tangential case, and hence
\[
 A_{R,T}(y,p)=0.
\]
If $|y_\perp|<R$, set $\rho:=\sqrt{R^2-|y_\perp|^2}$. The condition $y+tp\in B_R$ is then equivalent to $ |y\cdot e+t|p||<\rho$, that is,
\[
 \frac{-\rho-y\cdot e}{|p|}  <  t  <  \frac{\rho-y\cdot e}{|p|}.
\]
Thus the set of all $t\in\R$ for which $y+tp\in B_R$ is an interval of length
\[
 \frac{2\rho}{|p|}  =  \frac{2\sqrt{R^2-|y_\perp|^2}}{|p|}.
\]
Restricting this interval to $[0,T]$, we obtain
\bq \label{eq:crossing-time}
 A_{R,T}(y,p)  \leq  \min\lt\{  T,\,  \frac{2\sqrt{(R^2-|y_\perp|^2)_+}}{|p|}  \rt\}  \leq  \min\lt\{T,\frac{2R}{|p|}\rt\}.
\eq
In particular,
\bq \label{eq:crossing-moment}
 |p|^3A_{R,T}(y,p)  \leq  \min\{T|p|^3,2R|p|^2\}  \leq  2R|p|^2.
\eq
For $p=0$, we simply have
\[
 A_{R,T}(y,0)=T\one_{B_R}(y),
\]
hence \eqref{eq:crossing-moment} remains valid.

We now estimate separately the two terms in the mild formula. For the contribution of the initial datum, Tonelli's theorem, the change of variables $y=x-tp$, and $e^{-t}\leq1$ give
\begin{align*}
 \int_0^T\int_{B_R}\int_{\R^3}  |p|^3e^{-t}f_0(x-tp,p)\,\dpp\dx\dt  &=  \iint_{\R^3\times\R^3}  |p|^3f_0(y,p)  \lt(  \int_0^T e^{-t}\one_{B_R}(y+tp)\,\dt  \rt)  \,\dy\dpp  \\
 & \leq  \iint_{\R^3\times\R^3}  |p|^3f_0(y,p)A_{R,T}(y,p)\,\dy\dpp  \\
 & \leq  2R  \iint_{\R^3\times\R^3}  |p|^2f_0(y,p)\,\dy\dpp,
\end{align*}
where the last inequality follows from \eqref{eq:crossing-moment}.  Notice that this bound is uniform in $T$: a fixed free trajectory can spend at most $2R/|p|$ units of time in $B_R$, independently of the length of the observation interval.

We next consider the source term.  By Tonelli's theorem,
\begin{align*}
 I_G  &:=  \int_0^T\int_{B_R}\int_{\R^3}|p|^3  \int_0^t e^{-(t-s)}  G(s,x-(t-s)p,p)\,\ds  \,\dpp\dx\dt  \\
 &=  \int_0^T\int_s^T\int_{B_R}\int_{\R^3}  |p|^3e^{-(t-s)}  G(s,x-(t-s)p,p)  \,\dpp\dx\dt\ds.
\end{align*}
For fixed $s$, set
\[
 \tau:=t-s,  \quad  y:=x-\tau p.
\]
Then $0\leq\tau\leq T-s$, while $x\in B_R$ is equivalent to $y+\tau p\in B_R$.
Thus,
\begin{align*}
 I_G  &=  \int_0^T  \iint_{\R^3\times\R^3}  |p|^3G(s,y,p)  \lt(  \int_0^{T-s}  e^{-\tau}\one_{B_R}(y+\tau p)\,\dd\tau  \rt)  \,\dy\dpp\ds  \\
 &\leq  \int_0^T  \iint_{\R^3\times\R^3}  |p|^3G(s,y,p)  A_{R,T-s}(y,p)  \,\dy\dpp\ds  \\
 &\leq  2R  \int_0^T  \iint_{\R^3\times\R^3}  |p|^2G(s,y,p)\,\dy\dpp\ds,
\end{align*}
where we used \eqref{eq:crossing-moment}, with $T-s$ in place of $T$, in the last step.

Combining the two estimates gives
\[
\int_0^T\int_{B_R}\int_{\R^3}  |p|^3f(t,x,p)\,\dpp\dx\dt  \leq  2R\lt(  \iint_{\R^3\times\R^3}  |p|^2f_0(x,p)\,\dx\dpp  +  \int_0^T  \iint_{\R^3\times\R^3}  |p|^2G(s,x,p)\,\dx\dpp\ds  \rt),
\]
which is the desired estimate.

We emphasize that the geometric estimate \eqref{eq:crossing-moment} itself is uniform in $T$. The dependence on $T$ in the final bound enters only through the time integral of the source term.  For example, if
\[
 \sup_{0\leq s\leq T}  \iint_{\R^3\times\R^3}  |p|^2G(s,x,p)\,\dx\dpp  \leq C_T,
\]
then
\[
 \int_0^T\int_{B_R}\int_{\R^3}  |p|^3f(t,x,p)\,\dpp\dx\dt  \leq  2R\lt(  \iint_{\R^3\times\R^3}|p|^2f_0\,\dx\dpp  +TC_T  \rt).
\]
This completes the proof.
\end{proof}

\begin{remark}
The estimate \eqref{eq:crossing-time} uses the fact that a characteristic in $\R^3_x$ crosses a fixed bounded set only once.  On a periodic spatial domain the same characteristic may return to the set repeatedly, so this argument fails to yield the corresponding gain of third-moments.
\end{remark}

Applying Lemma \ref{lem:3-mom} to \eqref{eq:reg-mild} with $G=\calF^\e(f^\e)$ gives
\bq \label{eq:third-uni}
 \sup_{0<\e<\frac12}  \int_0^T\int_{B_R}\int_{\R^3}|p|^3f^\e(t,x,p)\,\dpp\dx\dt  \leq C_{T,R,\hbar,f_0}.
\eq

We next use the standard time-dependent velocity averaging lemma \cite{DLM91,GLPS88}.  Since $0\leq f^\e,\calF^\e(f^\e)\leq\frac1\hbar $, the uniform mass bounds imply
\[
 \sup_{0<\e<\frac12}\lt(  \|f^\e\|_{L^2((0,T)\times\R^3\times\R^3)}  +\|\calF^\e(f^\e)\|_{L^2((0,T)\times\R^3\times\R^3)}  \rt)<\infty.
\]
Hence, after extraction,
\bq \label{eq:weak-f}
 f^\e\rightharpoonup f  \quad\text{weakly in }L^2_{\rm loc}((0,\infty)\times\R^3\times\R^3)
\eq
and weakly-star in $L^\infty$.  In particular,
\bq \label{eq:lim-pauli}
 0\leq f\leq\frac1\hbar .
\eq
Since
\[
 (\pa_t+p\cdot\nabla_x)f^\e=\calF^\e(f^\e)-f^\e,
\]
velocity averaging yields, for every $\varphi\in C_c^\infty(\R^3_p)$,
\[
 \int_{\R^3}\varphi(p)f^\e(t,x,p)\,\dpp  \to  \int_{\R^3}\varphi(p)f(t,x,p)\,\dpp  \quad\text{strongly in }L^2_{\rm loc}((0,\infty)\times\R^3_x).
\]
The standard velocity-tail argument now extends this compactness to $1,p,|p|^2$: the density and momentum tails are controlled by the global second moment, whereas
\[
 \int_0^T\int_K\int_{|p|>L}|p|^2f^\e\,\dpp\dx\dt  \leq\frac{C_{T,K}}{L}
\]
for every compact $K\subset\R^3_x$, by \eqref{eq:third-uni}.  Consequently,
\[
 (N_{f^\e},P_{f^\e},E_{f^\e})  \to(N_f,P_f,E_f)  \quad\text{strongly in }L^1_{\rm loc}((0,\infty)\times\R^3_x).
\]
After a further extraction,
\bq \label{eq:mom-ae}
 (N_{f^\e},P_{f^\e},E_{f^\e})(t,x)  \to(N_f,P_f,E_f)(t,x)  \quad\text{for a.e. }(t,x).
\eq
The bound \eqref{eq:lim-pauli} and Lemma \ref{lem:sharp-pauli} imply that, for almost every $(t,x)$ with $N_f(t,x)>0$,
\[
 (N_f(t,x),P_f(t,x),E_f(t,x))\in\mathfrak M_\hbar.
\]
If $N_f(t,x)=0$, nonnegativity implies
\[
 (N_f(t,x),P_f(t,x),E_f(t,x))=(0,0,0).
\]
To upgrade the local convergence of $N_{f^\e}$ to global $L^1$ convergence, it remains to rule out loss of mass at spatial infinity.  Set
\[
 M_0:=\iint_{\R^3\times\R^3}f_0(x,p)\,\dx\dpp.
\]
By \eqref{eq:m-loss-s},
\[
 \int_0^T\int_{\R^3}N_{f^\e}(t,x)\,\dx\dt\to TM_0.
\]
Since $M^\e(t)\leq M^\e(0)=M_0$, for every $\tau\in(0,T)$ and $R>0$, the local convergence gives
\begin{align*}
 \int_\tau^T\int_{B_R}N_f(t,x)\,\dx\dt
 &=\lim_{\e\to0}  \int_\tau^T\int_{B_R}N_{f^\e}(t,x)\,\dx\dt\\
 &\geq  \lim_{\e\to0}  \int_0^T\int_{\R^3}N_{f^\e}(t,x)\,\dx\dt  -\tau M_0   -T\sup_{0<\e<\frac12}\sup_{0\leq t\leq T}  \int_{|x|>R}N_{f^\e}(t,x)\,\dx.
\end{align*}
Using \eqref{eq:spa-tig}, first letting $R\to\infty$ and then $\tau\downarrow0$ yields
\[
 \int_0^T\int_{\R^3}N_f(t,x)\,\dx\dt\geq TM_0.
\]
On the other hand, Fatou's lemma gives
\[
 \int_0^T\int_{\R^3}N_f(t,x)\,\dx\dt  \leq  \liminf_{\e\to0}  \int_0^T\int_{\R^3}N_{f^\e}(t,x)\,\dx\dt  =TM_0.
\]
Consequently,
\[
 \int_0^T\int_{\R^3}N_f(t,x)\,\dx\dt=TM_0.
\]
Since $N_{f^\e}\to N_f$ almost everywhere and the total integrals converge, Lemma \ref{lem:pos-L1} gives
\bq \label{eq:den-glo}
 N_{f^\e}\to N_f  \quad\text{strongly in }L^1((0,T)\times\R^3_x).
\eq

%
%
%
%
%
%

\subsection{Identification of the limit }

We now identify the limit in \eqref{eq:weak-f}.
\begin{lemma}
     The weak limit $f$ in \eqref{eq:weak-f} satisfies the mild formula \eqref{eq:mild-main}. Furthermore, up to the subsequence extracted so far, for every $T>0$ we have $f^\e \to f$ in $C([0,T];L^1(\R^3\times\R^3))$.
\end{lemma} 
\begin{proof}
We split the proof into two steps. First, we prove convergence of the local equilibria in the collision terms. Second, we prove that this is enough to pass to the limit in the mild formulation \eqref{eq:reg-mild}.  

\textit{Step 1.} Set
\[
 F^\e(t,x,p):=\calF^\e(f^\e)(t,x,p),  \quad  F(t,x,p):=\calF(f)(t,x,p).
\]
Fix a point $(t,x)$ outside the null set in \eqref{eq:mom-ae}.  If $N_f(t,x)>0$, Proposition \ref{prop:consist} and \eqref{eq:mom-ae} give
\[
 \|F^\e(t,x,\cdot)-F(t,x,\cdot)\|_{L^1_2(\R^3_p)}\to0.
\]
If $N_f(t,x)=0$, then $P_f=E_f=0$ and
\[
 \|F^\e(t,x,\cdot)\|_{L^1(\R^3_p)}  =\frac{N_{f^\e}(t,x)}{1+\e N_{f^\e}(t,x)}\to0,
\]
thus the same conclusion holds in unweighted $L^1(\R^3_p)$, which is all that is needed below.  Thus, for almost every $(t,x)$,
\bq \label{eq:pt-L1p}
 \int_{\R^3}|F^\e(t,x,p)-F(t,x,p)|\,\dpp\to0.
\eq

Define
\[
 h^\e(t,x):=\int_{\R^3}|F^\e(t,x,p)-F(t,x,p)|\,\dpp.
\]
By the local mass formulas,
\[
 h^\e(t,x) \leq N_{F^\e}(t,x)+N_F(t,x) =\frac{N_{f^\e}(t,x)}{1+\e N_{f^\e}(t,x)}+N_f(t,x) \leq N_{f^\e}(t,x)+N_f(t,x).
\]
By \eqref{eq:den-glo} and the arguments preceding it, we know that 
\[
 N_{f^\e} + N_f \to 2N_f  \quad\text{a.e. and strongly in }  L^1((0,T)\times\R^3).
\]
Since $h^\e\to0$ almost everywhere by \eqref{eq:pt-L1p}, this implies
\begin{equation*}
    N_{f^\e} + N_f - h^\e \to 2N_f \quad \text{a.e.}
\end{equation*}
Therefore, we may apply Fatou's lemma to $N_{f^\e} + N_f - h^\e \ge 0$ to obtain
\begin{align*}
    \int_0^T \int_{\R^3} 2N_f \, \dx \dt  \le \liminf_{\e \downarrow 0} \int_0^T \int_{\R^3} (N_{f^\e} + N_f - h^\e) \, \dx \dt = 2 \int_0^T \int_{\R^3} N_f \,\dx \dt - \limsup_{\e\downarrow 0} \int_0^T \int_{\R^3} h^\e \,\dx \dt,
\end{align*}
which means $h^\e \to 0$ in $L^1((0,T)\times \R^3)$.  That is, we have shown
\bq \label{eq:F-str-glo}
 F^\e\to F \quad\text{strongly in }L^1((0,T)\times\R^3\times\R^3)
\eq
for every $T>0$. 

\textit{Step 2.} We now pass directly to the mild formulation. Let us define
\[
 \widetilde f(t,x,p) :=e^{-t}f_0(x-tp,p) +\int_0^te^{-(t-s)}F(s,x-(t-s)p,p)\,\ds.
\]
From \eqref{eq:reg-mild}, the change of variables $y=x-(t-s)p$, and \eqref{eq:F-str-glo},
\[
 \sup_{0\leq t\leq T}  \|f^\e(t)-\widetilde f(t)\|_{L^1(\R^3\times\R^3)} \leq\int_0^T\iint_{\R^3\times\R^3}  |F^\e(s,y,p)-F(s,y,p)|\,\dy\dpp\ds \to0.
\]
Hence, $f^\e\to\widetilde f$ in $C([0,T];L^1)$.  Since $f^\e\rightharpoonup f$ by \eqref{eq:weak-f}, we have $\widetilde f=f$, and hence
\[
 f(t,x,p)  =e^{-t}f_0(x-tp,p)  +\int_0^te^{-(t-s)}\calF(f)(s,x-(t-s)p,p)\,\ds.
\]
This proves \eqref{eq:mild-main}. 
\end{proof}

Since $0\leq F\leq\frac1\hbar $, the formula also gives
\[
 0\leq f(t,x,p)  \leq e^{-t}\frac1\hbar   +\frac1\hbar \int_0^te^{-(t-s)}\,\ds  =\frac1\hbar .
\]
Finally, lower semicontinuity of the uniform second-moment estimates gives
\bq \label{eq:lim-L12-t}
 \int_0^T\|f(t)\|_{L^1_2}\,\dt<\infty.
\eq
The exact local moment matching \eqref{eq/ emm} for $\calF(f)$ then gives the same time-integrability for the equilibrium.

%
%
%
%
%
%

\section{Conservation laws and entropy inequality}\label{sec:H-thm}

We complete the proof of Theorem \ref{thm:main} by recovering the exact conservation laws and the entropy inequality.  By \eqref{eq:lim-L12-t} and exact moment matching \eqref{eq/ emm},
\[
 \int_0^T\lt(\|f(t)\|_{L^1_2}+\|\calF(f)(t)\|_{L^1_2}\rt) \dt<\infty  \quad\text{for every }T>0.
\]
Consequently, the mild formula implies $f(t)\in L^1_2$ for every $t\geq0$, and the weighted Fubini arguments below are justified.

We begin with the conservation laws, which follow from the exact local moment matching of the equilibrium.

\begin{lemma}\label{lem:conserv}
For every $t\geq0$,
\[
 \iint_{\R^3\times\R^3}(1,p,|p|^2)f(t,x,p)\,\dx\dpp  =\iint_{\R^3\times\R^3}(1,p,|p|^2)f_0(x,p)\,\dx\dpp.
\]
In particular,
\[
 \sup_{t\geq0}\|f(t)\|_{L^1_2}=\|f_0\|_{L^1_2}.
\]
\end{lemma}

\begin{proof}
 The mild formula \eqref{eq:mild-main} shows that
\begin{equation}\label{eq/ mild char}
		f(t,x+pt,p) = e^{-t} f_0(x,p) + \int_0^t e^{-(t-s)} \calF(f)(s,x+sp,p) \, \dd s.
\end{equation}
		We know from Proposition \ref{prop:vari} that \eqref{eq/ emm} holds for almost every $(t,x)$. Therefore, we just multiply \eqref{eq/ mild char} by $(1,p,|p|^2)$, integrate over $\R^3\times\R^3$, and apply Tonelli's theorem along with change of variables to complete the proof.
\end{proof}

The entropy argument also requires control of spatial tails.  The next estimate propagates the spatial second moment.

\begin{lemma}\label{lem:x2-propa}
Assume in addition that
\[
 X_0:=\iint_{\R^3\times\R^3}|x|^2f_0(x,p)\,\dx\dpp<\infty.
\]
Set
\[
 X(t):=\iint_{\R^3\times\R^3}|x|^2f(t,x,p)\,\dx\dpp,  \quad  \calE_0:=\iint_{\R^3\times\R^3}|p|^2f_0(x,p)\,\dx\dpp.
\]
Then
\bq \label{eq:x2-bd}
 X(t) \leq 2(X_0 + t^2\calE_0) \quad\text{for every }t\geq0.
\eq
\end{lemma}

\begin{proof}
Multiplying the formula \eqref{eq/ mild char} by $|x|^2$ and then integrating over $\R^3\times\R^3$, we obtain after a change of variables
\begin{align*}
    &\iint_{\R^3\times\R^3} |x-pt|^2 f(t,x,p) \, \dx \dpp\\
    &= \iint_{\R^3\times\R^3} e^{-t} |x|^2 f_0(x,p) \, \dx \dpp + \int_0^t \iint_{\R^3\times\R^3} e^{-(t-s)} |x-sp|^2 \calF(f)(s,x,p) \, \dd s \dx \dpp \\
    &= \iint_{\R^3\times\R^3} e^{-t} |x|^2 f_0(x,p) \, \dx \dpp + \int_0^t \iint_{\R^3\times\R^3} e^{-(t-s)} |x-sp|^2 f(s,x,p) \, \dd s \dx \dpp.
\end{align*}
Note that we applied Tonelli's theorem (all terms above are non-negative) as well as the moment-matching in \eqref{eq/ emm}. Solving this integral equation gives
\begin{align*}
    \iint_{\R^3\times\R^3} |x-pt|^2 f(t,x,p) \, \dx \dpp = \iint_{\R^3\times\R^3} |x|^2 f_0(x,p) \, \dx \dpp = X_0.
\end{align*}
Then
\begin{align*}
    \iint_{\R^3\times\R^3}|x|^2 f \, \dx \dpp \le 2\iint_{\R^3\times\R^3}|x-pt|^2 f \, \dx \dpp + 2t^2 \iint_{\R^3\times\R^3}|p|^2 f \, \dx \dpp \le 2(X_0 + t^2 \calE_0),
\end{align*}
the last inequality owing to the previous identity above, along with Lemma \ref{lem:conserv}. This proves \eqref{eq:x2-bd}.
\end{proof}

The spatial and kinetic moment bounds can now be converted into the entropy integrability needed for the H-theorem.

\begin{lemma}\label{lem:entr-integ}
Let $0\leq g\leq\frac1\hbar $ and assume
\[
 \iint_{\R^3\times\R^3}(1+|x|^2+|p|^2)g(x,p)\,\dx\dpp<\infty.
\]
Then
\bq \label{eq:entr-integ-bd}
 \iint_{\R^3\times\R^3}g|\log g|\,\dx\dpp  +\iint_{\R^3\times\R^3}|\eta_\hbar(g)|\,\dx\dpp  \leq C_\hbar\lt(  1+\iint_{\R^3\times\R^3}(1+|x|^2+|p|^2)g\,\dx\dpp  \rt).
\eq
\end{lemma}

\begin{proof}
Set $V(x,p):=1+|x|^2+|p|^2$.  Since $g\leq\frac1\hbar $, we obtain $ (g\log g)^+\leq C_\hbar g$. On $\{g\geq e^{-V}\}$, we get $ (g\log g)^-\leq Vg$. On $\{g<e^{-V}\}$, we have $e^{-V}\leq e^{-1}$ and the map $z\mapsto-z\log z$ is increasing on $(0,e^{-1}]$; hence
\[
 (g\log g)^-\leq Ve^{-V}.
\]
Since $Ve^{-V}\in L^1(\R^3\times\R^3)$, these estimates give the first term in \eqref{eq:entr-integ-bd}.  Finally, for $0\leq z\leq1$, we get $|(1-z)\log(1-z)|\leq z$, and thus
\[
 \lt|\frac1\hbar(1-\hbar g)\log(1-\hbar g)\rt|\leq g.
\]
This proves \eqref{eq:entr-integ-bd}.
\end{proof}

We now prove the entropy statement in Theorem \ref{thm:main}.  Recall the definitions of $\mathscr H_\hbar$ and $\mathscr D_\hbar$ from \eqref{eq:H-def}--\eqref{eq:D-def}.  By Lemmas \ref{lem:conserv}, \ref{lem:x2-propa}, and \ref{lem:entr-integ}, for every $T>0$,
\bq \label{eq:entr-unif-f}
 \sup_{0\leq t\leq T}  \iint_{\R^3\times\R^3}  \lt(f(t,x,p)|\log f(t,x,p)|+|\eta_\hbar(f(t,x,p))|\rt) \,\dx\dpp<\infty.
\eq
The same estimate holds with $f$ replaced by $\calF(f)$, since the two functions have the same local mass and kinetic energy, which means
\[
 \iint_{\R^3\times\R^3}|x|^2\calF(f)\,\dx\dpp  =\iint_{\R^3\times\R^3}|x|^2f\,\dx\dpp,  \quad  \iint_{\R^3\times\R^3}|p|^2\calF(f)\,\dx\dpp  =\calE_0.
\]

Since $\eta_\hbar'$ has logarithmic singularities at both endpoints of the Pauli interval, we first replace the entropy density with a suitably bounded approximation. It is more convenient to truncate $\eta_{\hbar}'$ instead of $\eta_{\hbar}$ itself, in order to preserve its monotonicity and hence the convexity of the approximating entropy density. For $m\geq1$ we set
\[
 \Theta_m(r):=\max\{-m,\min\{r,m\}\}
\]
and
\[
 \eta_{\hbar,m}(z)  :=\int_0^z\Theta_m \lt(\eta_\hbar'(s)\rt)\,\ds,  \quad  \eta_\hbar'(s)=\log\frac{s}{1-\hbar s},  \quad 0<s<\frac1\hbar .
\]
Since $\eta_\hbar'$ and $\Theta_m$ are nondecreasing, $\eta_{\hbar,m}$ is convex; moreover $\eta_{\hbar,m}\in C^1([0,\frac1\hbar ])$ and $|\eta_{\hbar,m}'|\leq m$.

Now
\[
 |\eta_\hbar'(s)|  \le |\log s|+|\log(1-\hbar s)|,
\]
and the right-hand side is a member of $L^1(0,\frac{1}{\hbar})$; therefore, the estimate $|\Theta_m(r)|\le |r|$ and the dominated convergence theorem show
\begin{equation*}
    \forall z\in [0,\frac{1}{\hbar}], \quad  \eta_{\hbar,m}(z) \xrightarrow[m\to\infty]{} \int_0^z \eta_{\hbar}'(s) \, \ds = \eta_h(z).
\end{equation*}
Moreover, we may compute again using $|\Theta_m(r)|\leq|r|$:
\bq \label{eq:entr-tru-dom}
 \forall z\in [0,\frac{1}{\hbar}], \quad  |\eta_{\hbar,m}(z)| \leq\int_0^z\lt(|\log s|+|\log(1-\hbar s)|\rt)\,\ds \leq C_\hbar z(1+|\log z|) .
\eq

We now compute the approximate entropy identities for the solutions. Define
\[
 \widetilde f(t,y,p):=f(t,y+tp,p),  \quad  \widetilde F(t,y,p):=\calF(f)(t,y+tp,p).
\]
The mild formula gives
\[
 \widetilde f(t,y,p)  =e^{-t}f_0(y,p)+\int_0^te^{-(t-s)}\widetilde F(s,y,p)\,\ds.
\]
Since $\widetilde F\in L^1((0,T)\times\R^3\times\R^3)$, Fubini's theorem shows that for almost every $(y,p)$, the map $t\mapsto\widetilde f(t,y,p)$ is absolutely continuous on $[0,T]$ and
\[
\pa_t\widetilde f=\widetilde F-\widetilde f  \quad\text{for a.e. }t\in(0,T).
\]
Thus, the ordinary chain rule gives
\[
 \eta_{\hbar,m}(\widetilde f(t,y,p))  -\eta_{\hbar,m}(f_0(y,p))  =\int_0^t\eta_{\hbar,m}'(\widetilde f)  (\widetilde F-\widetilde f)\,\ds.
\]
Since $|\eta_{\hbar,m}'(\widetilde f)(\widetilde F-\widetilde f)|\leq m(\widetilde F+\widetilde f)$, we may integrate over $\R^3_y\times\R^3_p$.  The measure-preserving change of variables $x=y+sp$ gives
\bq \label{eq:trun-cha}
 \begin{aligned}
 &\iint_{\R^3\times\R^3}\eta_{\hbar,m}(f(t,x,p))\,\dx\dpp  -\iint_{\R^3\times\R^3}\eta_{\hbar,m}(f_0(x,p))\,\dx\dpp\\
 &\quad=  \int_0^t\iint_{\R^3\times\R^3}  \eta_{\hbar,m}'(f(s,x,p))  \lt(\calF(f)(s,x,p)-f(s,x,p)\rt)\,\dx\dpp\ds.
 \end{aligned}
\eq
By convexity of $\eta_{\hbar,m}$,
\[
 \eta_{\hbar,m}'(f)(\calF(f)-f)  \leq\eta_{\hbar,m}(\calF(f))-\eta_{\hbar,m}(f).
\]
Substitution in \eqref{eq:trun-cha} yields
\bq \label{eq:trun-H}
 \begin{aligned}
 &\iint_{\R^3\times\R^3}\eta_{\hbar,m}(f(t,x,p))\,\dx\dpp +\int_0^t\iint_{\R^3\times\R^3}  \lt(\eta_{\hbar,m}(f(s,x,p)) -\eta_{\hbar,m}(\calF(f)(s,x,p))\rt) \,\dx\dpp\ds\\
  &\quad \leq\iint_{\R^3\times\R^3}\eta_{\hbar,m}(f_0(x,p))\,\dx\dpp.
 \end{aligned}
\eq
By \eqref{eq:entr-tru-dom},
\[
 |\eta_{\hbar,m}(f)|  +  |\eta_{\hbar,m}(\calF(f))|  \leq  C_\hbar  \lt[  f(1+|\log f|)  +  \calF(f)\lt(1+|\log\calF(f)|\rt)  \rt].
\]
The right-hand side belongs to
\[
 L^1((0,T)\times\R^3\times\R^3)
\]
by \eqref{eq:entr-unif-f} and the corresponding estimate for $\calF(f)$. Thus, dominated convergence in \eqref{eq:trun-H} as $m\to\infty$ gives
\[
 \mathscr H_\hbar(f(t))  +\int_0^t\mathscr D_\hbar(f(s))\,\ds  \leq\mathscr H_\hbar(f_0).
\]
By Proposition \ref{prop:vari}, the local entropy gap defining $\mathscr D_\hbar$ is nonnegative for almost every $(t,x)$.  Moreover,
\[
 0\leq\mathscr D_\hbar(f(t))  \leq\iint_{\R^3\times\R^3}  \lt(|\eta_\hbar(f)|+|\eta_\hbar(\calF(f))|\rt)\,\dx\dpp,
\]
thus $\mathscr D_\hbar(f)\in L^1(0,T)$ by \eqref{eq:entr-unif-f}.  This proves \eqref{eq:H-thm-gap} and completes the proof of Theorem \ref{thm:main}.

%
%
%
%
%
%
\section{Semiclassical limit}\label{sec:semiclas}

We now let the quantum parameter tend to zero. In this section the dependence on $\hbar$ is displayed explicitly.  For $\hbar\geq0$, set
\[
 \beta_{k}^{\hbar}(c)  :=\int_{\R^3}\frac{|p|^k}{e^{|p|^2+c}+\hbar}\,\dpp,  \quad  \beta^{\hbar}(c)  :=\frac{\beta_0^{\hbar}(c)}{(\beta_2^{\hbar}(c))^{\frac35}}.
\]
For $\hbar=0$, this means
\[
 \beta_k^0(c)=e^{-c}\int_{\R^3}|p|^ke^{-|p|^2}\,\dpp.
\]

The first ingredient in the semiclassical identification is the convergence of the scalar Fermi--Dirac moment map to its classical counterpart.

\begin{lemma}\label{lem:beta-semic}
For every compact interval $I\Subset\R$,
\[
 \sup_{c\in I}|\beta^{\hbar}(c)-\beta^0(c)|\to0  \quad\text{as }\hbar\downarrow0,
\]
where
\bq \label{eq:beta-clas}
 \beta^0(c)  =e^{-2c/5}\lt(\frac{2\pi}{3}\rt)^{\frac35}.
\eq
Moreover, for $k=0,2$, $\beta_k^{\hbar}\to\beta_k^0$ uniformly on every compact interval.
\end{lemma}

\begin{proof}
Fix $A>0$ and $|c|\leq A$.  For $k=0,2$,
\[
 |\beta_k^{\hbar}(c)-\beta_k^0(c)| =\int_{\R^3} |p|^k\frac{\hbar}{e^{|p|^2+c}(e^{|p|^2+c}+\hbar)}\,\dpp \leq \hbar e^{2A}\int_{\R^3}|p|^ke^{-2|p|^2}\,\dpp.
\]
Thus for every $k\ge 0$, we have $\beta_k^{\hbar}\to\beta_k^0$ uniformly on $[-A,A]$. Using these convergences, it can be straightforwardly checked that $\beta^{\hbar} \to \beta^0 := \frac{\beta^0_0}{(\beta^0_2)^{3/5}}$ uniformly on $[-A,A]$, since it is easy to see that $\beta_2^{\hbar}$ is bounded away from zero and $\beta_0^{\hbar}$ is upper-bounded (uniformly in small $\hbar$).

To observe the identity \eqref{eq:beta-clas} for $\beta^0$ we just need to compute
\[
 \int_{\R^3}e^{-|p|^2}\,\dpp=\pi^{\frac32},  \quad  \int_{\R^3}|p|^2e^{-|p|^2}\,\dpp=\frac32\pi^{\frac32}.
\]

\end{proof}

We next recall a version of the velocity averaging lemma for $L^1$ source terms; see also \cite{GSR02}. In \cite[Proposition 4.1]{KS25}, an averaging lemma was stated for a family $(h_n)$ sharing the same initial data. However, following their proof, it is straightforward to observe that that assumption may be replaced with (v) as written below.

\begin{lemma}\label{lem:L1-aver}
Let $q\in(1,\infty)$, and let $h_n\geq0$ solve
\[
\pa_t h_n+p\cdot\nabla_x h_n=G_n^+-G_n^-,  \quad G_n^\pm\geq0.
\]
Assume on every finite time interval that
\begin{enumerate}[label=(\roman*)]
 \item $\{h_n\}$ is bounded in $L^\infty_tL^q_{x,p}$;
 \item $\{(1+|p|^2)h_n\}$ is bounded in $L^\infty_tL^1_{x,p}$;
 \item $\{G_n^\pm\}$ is bounded in $L^1_{t,x,p}$;
 \item the densities $\int h_n\,\dpp$ are uniformly tight in $x$, uniformly in time;
 \item $h_n(0)$ is relatively compact in $L^q(\R^3\times\R^3)$.
\end{enumerate}
Then for every $\psi\in C_c(\R^3_p)$, the averages
\[
 \int_{\R^3}\psi(p)h_n(t,x,p)\,\dpp
\]
are relatively compact in $L^1_{\mathrm{loc}}([0,\infty);L^1(\R^3_x))$.
\end{lemma}

Let us now prove Theorem \ref{thm:semi}. We remind the reader that for a nonnegative distribution $g$ with finite second moments, we denote by $M(g)$ the local Maxwellian:
\[
 M(g)(t,x,p)=
 \begin{cases}
 \displaystyle  \frac{N_g}{(2\pi T_g)^{\frac32}} \exp \lt(-\frac{|p-u_g|^2}{2T_g}\rt),&N_g>0,\\[8pt]
 0,&N_g=0.
 \end{cases}
\]
On the non-vacuum set $N_g>0$, the coefficients of the local Maxwellian are defined by
\[
 u_g=\frac{P_g}{N_g},  \quad  T_g=\frac{E_g-|P_g|^2/N_g}{3N_g},
\]
whereas on $N_g = 0$ we define $u_g, T_g$ to be zero.

\begin{proof}[Proof of Theorem \ref{thm:semi}]

We divide the proof into several steps.

\emph{Step 1: uniform moment and entropy bounds.}
On one hand, we know from Theorem \ref{thm:main} that the mass and kinetic energy bounds \eqref{eq:semi-ini-b} for the initial data are conserved by the solutions $(f_n)$ to the Fermi--Dirac BGK equation. On the other hand, we know from Lemma \ref{lem:x2-propa} that the spatial second moments may be controlled uniformly in the quantum parameter (and thus $n$). In conclusion, for every $T>0$ we have
\bq \label{eq:semi-mom}
 \sup_n\sup_{0\leq t\leq T}  \iint_{\R^3\times\R^3}(1+|x|^2+|p|^2)f_n(t,x,p)\,\dx\dpp  \leq C_T.
\eq
The entropy statement of Theorem \ref{thm:main} applies to each $f_n$.  Since
\[
 -f_n\leq  \frac1{\hbar_n}(1-\hbar_nf_n)\log(1-\hbar_nf_n)  \leq0,
\]
we have
\begin{align*}
 \iint_{\R^3\times\R^3} f_n(t)\log f_n(t)\,\dx\dpp  &\leq \mathscr H_{\hbar_n}(f_n(t))  +\iint_{\R^3\times\R^3} f_n(t)\,\dx\dpp\\
 &\leq \mathscr H_{\hbar_n}(f_{0,n})+M_{0,n} \leq \iint_{\R^3\times\R^3} f_{0,n}\log f_{0,n}\,\dx\dpp+M_{0,n}  \leq C.
\end{align*}
The negative part of $f_n\log f_n$ is uniformly bounded by classical arguments using the second moments (see the proof of Lemma \ref{lem:entr-integ}). Hence
\bq \label{eq:semi-LlogL}
 \sup_n\sup_{0\leq t\leq T}  \iint_{\R^3\times \R^3} f_n(t,x,p)|\log f_n(t,x,p)|\,\dx\dpp  \leq C_T.
\eq
Thus the family $\{f_n\}$ is uniformly integrable over $(0,T)\times\R^3\times\R^3$. Together with the moment bounds, the Dunford--Pettis theorem yields modulo a (diagonal) subsequence
\bq \label{eq:semi-weak}
 \forall T>0, \quad f_n\rightharpoonup f  \quad\text{weakly in }L^1((0,T)\times\R^3\times\R^3).
\eq

\emph{Step 2: compactness of velocity averages.}  To recover compactness without an $\hbar_n$-uniform $L^\infty$ bound, we use a bounded renormalization of the distribution.  We remind the reader that such nonlinear renormalizations are classical in the weak compactness theory for kinetic equations \cite{DL89,Mis10,MP97}, and are often used to bypass the problem that averaging lemmas do not apply to merely $L^1$-bounded families of solutions (see (i) of Lemma \ref{lem:L1-aver}).

Fix $\delta>0$, and define
\[
 \Gamma_\delta(z):=\frac{z}{1+\delta z},  \quad  h_{n,\delta}:=\Gamma_\delta(f_n).
\]
Since
\[
 \Gamma_\delta'(z)=\frac1{(1+\delta z)^2},
\]
we have
\bq \label{eq:Gam-deri-bd}
 0\leq \Gamma_\delta'(z)\leq1  \quad\text{for all }z\geq0.
\eq
Since $f_n$ is a mild solution, the chain rule along characteristics gives
\bq \label{eq:semi-renorm}
 (\pa_t+p\cdot\nabla_x)h_{n,\delta}  =  G_{n,\delta}^+-G_{n,\delta}^-,
\eq
where
\[
G_{n,\delta}^+  :=  \Gamma_\delta'(f_n)\calF^{\hbar_n}(f_n),  \quad  G_{n,\delta}^-  :=  \Gamma_\delta'(f_n)f_n.
\]

Now let us check that the conditions of Lemma \ref{lem:L1-aver} apply to $\{h_{n,\delta}\}_n$ (for each $\delta>0$). Since $f_n\geq0$,
\bq \label{eq:h-delta-pt}
 0\leq h_{n,\delta}  \leq \min\{f_n,\delta^{-1}\}.
\eq
In particular, for every $t\geq0$,
\[
 \iint_{\R^3\times\R^3}  h_{n,\delta}(t,x,p)\,\dx\dpp  \leq  \iint_{\R^3\times\R^3}  f_n(t,x,p)\,\dx\dpp  =  \iint_{\R^3\times\R^3}  f_{0,n}(x,p)\,\dx\dpp  =:M_{0,n}.
\]
By the assumptions on the initial data,
\[
 \sup_n M_{0,n}<\infty.
\]
Thus, for every fixed $\delta>0$ and every $1<q<\infty$,
\[
 \|h_{n,\delta}(t)\|_{L^q(\R^3\times\R^3)}^q  =  \iint_{\R^3\times\R^3}  h_{n,\delta}^q\,\dx\dpp  \leq  \delta^{-(q-1)}  \iint_{\R^3\times\R^3}  h_{n,\delta}\,\dx\dpp  \leq  C\delta^{-(q-1)}.
\]
Consequently, for every $T>0$,
\[
 \sup_n  \|h_{n,\delta}\|_{L^\infty(0,T;L^q(\R^3\times\R^3))}  \leq C_{\delta,q}.
\]

Moreover, since $h_{n,\delta}\leq f_n$, it is clear that \eqref{eq:semi-mom} implies
\[
 \sup_n\sup_{0\leq t\leq T}  \iint_{\R^3\times\R^3}  |p|^2h_{n,\delta}(t,x,p)\,\dx\dpp  \leq C_T.
\]
It is also clear that each of the gain and loss source terms in \eqref{eq:semi-renorm} are uniformly bounded in $L^1((0,T)\times\R^3\times\R^3)$. Indeed, on one hand, \eqref{eq:Gam-deri-bd} gives
\[
 0\leq G_{n,\delta}^-\leq f_n,
\]
and hence
\[
 \int_0^T\iint_{\R^3\times\R^3}  G_{n,\delta}^-\,\dx\dpp\dt  \leq  \int_0^T\iint_{\R^3\times\R^3}  f_n\,\dx\dpp\dt  \leq C T.
\]
On the other hand,
\[
 0\leq  G_{n,\delta}^+  \leq  \calF^{\hbar_n}(f_n).
\]
Since the Fermi--Dirac equilibrium has the same local mass as $f_n$:
\[
 \int_{\R^3}  \calF^{\hbar_n}(f_n)(t,x,p)\,\dpp  =  N_{f_n}(t,x),
\]
we deduce
\[
 \int_0^T\iint_{\R^3\times\R^3}  G_{n,\delta}^+\,\dx\dpp\dt  \leq  \int_0^T\iint_{\R^3\times\R^3}  \calF^{\hbar_n}(f_n)\,\dx\dpp\dt =  \int_0^T\iint_{\R^3\times\R^3}  f_n\,\dx\dpp\dt  \leq CT.
\]

We also have spatial tightness for $h_{n,\delta}$.  Indeed, by \eqref{eq:h-delta-pt},
\[
 \int_{\R^3}h_{n,\delta}(t,x,p)\,\dpp  \leq N_{f_n}(t,x),
\]
and hence \eqref{eq:semi-mom} gives
\[
 \sup_n\sup_{0\leq t\leq T}  \int_{|x|>R}\int_{\R^3}  h_{n,\delta}(t,x,p)\,\dpp\dx  \leq \frac{C_T}{R^2}.
\]

It remains to verify compactness of the initial data for the renormalized equation.  Since $\Gamma_\delta$ is one-Lipschitz by \eqref{eq:Gam-deri-bd}, \eqref{eq:semi-ini-conv} yields
\[
 \Gamma_\delta(f_{0,n}) \to  \Gamma_\delta(f_0)  \quad\text{strongly in }  L^1(\R^3\times\R^3).
\]
Moreover,
\[
 0\leq  \Gamma_\delta(f_{0,n}),  \Gamma_\delta(f_0)  \leq\delta^{-1},
\]
so for every $1<q<\infty$,
\[
 \|\Gamma_\delta(f_{0,n})        -\Gamma_\delta(f_0)\|_{L^q(\R^3\times\R^3)}^q \leq  \lt(\frac{2}{\delta}\rt)^{q-1}  \|\Gamma_\delta(f_{0,n})        -\Gamma_\delta(f_0)\|_{L^1(\R^3\times\R^3)} \to0.
\]
Thus
\[
 \Gamma_\delta(f_{0,n}) \to  \Gamma_\delta(f_0)  \quad\text{strongly in }  L^q(\R^3\times\R^3)  \quad(1<q<\infty).
\]

Then, we may apply Lemma \ref{lem:L1-aver} to \eqref{eq:semi-renorm}.  For every $\psi\in C_c(\R^3_p)$, the sequence
\bq \label{eq:reno-av}
 \lt\{  \int_{\R^3}  \psi(p)h_{n,\delta}(t,x,p)\,\dpp  \rt\}_n  =  \lt\{  \int_{\R^3}  \psi(p)\Gamma_\delta(f_n)(t,x,p)\,\dpp  \rt\}_n
\eq
is relatively compact in $ L^1_{\rm loc}([0,\infty)\times\R^3_x)$ for every fixed $\delta>0$.

We now remove the renormalization.  Since
\[
 0\leq  f_n-\Gamma_\delta(f_n)  =  \frac{\delta f_n^2}{1+\delta f_n}  \leq f_n,
\]
fix $A>e$ and split the phase space into $\{f_n\leq A\}$ and $\{f_n>A\}$.  On the first set,
\[
 \frac{\delta f_n^2}{1+\delta f_n}  \leq\delta f_n^2  \leq\delta A f_n.
\]
On the second set we simply use $f_n-\Gamma_\delta(f_n)\leq f_n$.  Hence
\bq \label{eq:reno-spl}
\int_0^T  \iint_{\R^3\times\R^3}  \lt(f_n-\Gamma_\delta(f_n)\rt) \dx\dpp\dt   \leq  \delta A  \int_0^T  \iint_{\R^3\times\R^3}  f_n\,\dx\dpp\dt  +  \int_0^T  \iint_{\{f_n>A\}}  f_n\,\dx\dpp\dt.
\eq
Since $A>e$, on $\{f_n>A\}$ we have
\[
 f_n  \leq  \frac{f_n\log f_n}{\log A},
\]
and thus \eqref{eq:semi-LlogL} gives
\[
 \int_0^T  \iint_{\{f_n>A\}}f_n\,\dx\dpp\dt  \leq  \frac{C_T}{\log A}.
\]
Combining this with the uniform mass bound in \eqref{eq:reno-spl}, we obtain
\[
 \sup_n  \|f_n-\Gamma_\delta(f_n)\|_  {L^1((0,T)\times\R^3_x\times\R^3_p)}  \leq  C_T\lt(  \delta A+\frac1{\log A}  \rt).
\]
For $0<\delta<e^{-2}$, choose $ A=\delta^{-\frac12}$. Then
\[
 \delta A=\delta^{\frac12},  \quad  \frac1{\log A}  =\frac{2}{|\log\delta|},
\]
and hence
\bq \label{eq:rem-reno}
 \lim_{\delta\downarrow0}  \sup_n  \|f_n-\Gamma_\delta(f_n)\|_  {L^1((0,T)\times\R^3\times\R^3)}  =0.
\eq

We finally transfer the compactness from the renormalized averages to the original ones.  Let $\psi\in C_c(\R^3_p)$, $T>0$, and let $K\subset\R^3_x$ be compact.  For every $n,m$,
\begin{align*}
 \lt\|  \int_{\R^3}\psi(p)  \lt(f_n-f_m\rt)(t,x,p)\,\dpp  \rt\|_{L^1((0,T)\times K)}
 & \leq  \|\psi\|_{L^\infty}  \|f_n-\Gamma_\delta(f_n)\|_  {L^1((0,T)\times\R^3\times\R^3)}  \\
 &\quad+  \lt\|  \int_{\R^3}\psi(p)  \lt(  \Gamma_\delta(f_n)-\Gamma_\delta(f_m)  \rt)\,\dpp  \rt\|_{L^1((0,T)\times K)}  \\
 &\quad+ \|\psi\|_{L^\infty} \|f_m-\Gamma_\delta(f_m)\|_{L^1((0,T)\times\R^3\times\R^3)}.
\end{align*}
Given any $\eta>0$, we first choose $\delta>0$ so small that the first and third terms are uniformly smaller than $\eta$, using \eqref{eq:rem-reno}.  For this fixed $\delta$, the middle term is compact by \eqref{eq:reno-av}.  This proves that
\[
 \lt\{  \int_{\R^3}  \psi(p)f_n(t,x,p)\,\dpp  \rt\}_n
\]
is relatively compact in $ L^1_{\rm loc}([0,\infty)\times\R^3_x)$ for every $\psi\in C_c(\R^3_p)$.

\emph{Step 3: strong compactness of the macroscopic moments.}
Once Step 2 provides compactness of the velocity averages against compactly supported functions of $p$, we can use the same velocity-tail argument as in Section \ref{sec:pass-eps}.  The uniform kinetic-energy bound \eqref{eq:semi-mom} controls the tails of the density and momentum.  For the energy moment, Lemma \ref{lem:3-mom}, applied to the mild equation for $f_n$ and using the exact kinetic-energy matching of $\calF^{\hbar_n}(f_n)$, gives, for every $T,R>0$,
\[
 \sup_n  \int_0^T\int_{B_R}\int_{\R^3}  |p|^3f_n(t,x,p)\,\dpp\dx\dt  \leq C_{T,R}.
\]
Consequently,
\[
 \sup_n  \int_0^T\int_{B_R}\int_{|p|>L}  |p|^2f_n(t,x,p)\,\dpp\dx\dt  \leq\frac{C_{T,R}}{L}.
\]
Combining these tail estimates with the compactness obtained in Step 2 yields, after extraction,
\[
 (N_{f_n},P_{f_n},E_{f_n})  \to  (N_f,P_f,E_f)  \quad\text{strongly in }  L^1_{\rm loc}((0,T)\times\R^3_x),
\]
and hence almost everywhere in $(t,x)$.

The spatial second-moment bound in \eqref{eq:semi-mom} gives uniform spatial tightness of the densities.  Together with exact mass conservation and \eqref{eq:semi-ini-conv}, the argument leading to \eqref{eq:den-glo} gives
\bq \label{eq:semi-den-glo}
 N_{f_n}\to N_f
 \quad\text{strongly in }L^1((0,T)\times\R^3_x).
\eq
 
\emph{Step 4: identification of the limit of the  quantum equilibria.}
Fix a point $(t,x)$ at which the moment convergence holds. On the vacuum set $N_f(t,x)=0$, we merely have
\[
 \|\calF^{\hbar_n}(f_n)(t,x,\cdot)\|_{L^1(\R^3_p)}  =N_{f_n}(t,x)\to0.
\]
Since $M(f)(t,x,\cdot)$ is defined to be zero when $N_f(t,x)=0$, we have proven that $\|\calF^{\hbar_n}(f_n)(t,x,\cdot) - M(f)(t,x,\cdot)\|_{L^1(\R^3_p)} \to 0$ for every point $(t,x)$ where $N_f(t,x)=0$.

Assume now $N_f(t,x)>0$.  Then
\[
 I_f:=E_f-\frac{|P_f|^2}{N_f}>0.
\]
Indeed, if $I_f=0$, then
\[
 \int_{\R^3}\lt|p-\frac{P_f}{N_f}\rt|^2f(t,x,p)\,\dpp=0,
\]
and thus the $L^1(\R^3_p)$-function $f(t,x,\cdot)$ is supported on a single point, hence vanishes almost everywhere, contradicting $N_f>0$.
Consequently,
\begin{equation} \label{eq/ bn1}
 B_n:=\frac{N_{f_n}}{(E_{f_n}-|P_{f_n}|^2/N_{f_n})^{\frac35}}  \to  B:=\frac{N_f}{I_f^{\frac35}}\in(0,\infty).
\end{equation}
On the other hand,
\begin{equation}\label{eq/ bn2}
 \beta^{\hbar_n}(-\infty) = \frac{5^{\frac35}(4\pi)^{\frac25}}{3\hbar_n^{\frac25}}  \to\infty.
\end{equation}
Eqs. \eqref{eq/ bn1}--\eqref{eq/ bn2} imply that
\begin{equation*}
    \sup \lt\{ n\in \mathbb{N} : B_n = \beta^{\hbar_n}(-\infty) \rt\}  < \infty,
\end{equation*} 
in other words for all sufficiently large $n$, the moment triple of $f_n(t,x,\cdot)$ lies in the interior branch $B_n<\beta^{\hbar_n}(-\infty)$. In
particular, for all large $n$ there is a finite $c_n\in \R$ such that
\[
 \beta^{\hbar_n}(c_n)=B_n.
\]

We now prove that $c_n\to  (\beta^0)^{-1}(B)$. First, we show that $\{c_n\}$ is compact in $\R$. Since the range of $\beta^0$ is $(0,\infty)$, we may choose two real numbers $c_-<c_+$ such that
\[
 \beta^0(c_-)>2B,  \quad  \beta^0(c_+)<\frac{B}{2}.
\]
Thanks to the local uniform convergence result $\beta^{\hbar}\to \beta^0$ in Lemma \ref{lem:beta-semic}, as well as the fact that $\beta^{\hbar_n}(c_n)\to B$ (see \eqref{eq/ bn1}), there holds for all large $n$ that
\[
 \beta^{\hbar_n}(c_-)>\frac32B,  \quad  \beta^{\hbar_n}(c_+)<\frac34B,  \quad  \frac34B< \beta^{\hbar_n}(c_n) <\frac32B.
\]
In particular
\begin{equation*}
    \beta^{\hbar_n}(c_+) < \beta^{\hbar_n}(c_n) < \beta^{\hbar_n}(c_-). 
\end{equation*}
Fixing any suitable $n$, we know that the function $\beta^{\hbar_n}$ is strictly decreasing, so
\[
 c_-<c_n<c_+.
\]
Thus $\{c_n\}$ is relatively compact in $\R$. Second, let us identify the limit. Consider any convergent subsequence $c_{n_j}\to c_*$. Then the local uniform convergence result of Lemma \ref{lem:beta-semic} shows
\begin{align*}
 |\beta^{\hbar_{n_j}}(c_{n_j}) - \beta^0(c_*)| \le |\beta^{\hbar_{n_j}}(c_{n_j}) - \beta^0(c_{n_j})| + |\beta^0(c_{n_j}) - \beta^0(c_*)| \to 0.
\end{align*}
That is, $\beta^{\hbar_{n_j}}(c_{n_j})\to \beta^0(c_*)$. However we also know from \eqref{eq/ bn1} that $\beta^{\hbar_{n_j}}(c_{n_j}) = B_{n_j} \to B$. Therefore it must be that $c_* = (\beta^0)^{-1}(B)$, and the uniqueness of the limit shows that the full sequence converges:
\[
 c_n\to c :=(\beta^0)^{-1}(B).
\]

Furthermore, we know that (pointwisely)
\[
 u_n:=\frac{P_{f_n}}{N_{f_n}}\to u:=\frac{P_f}{N_f},
\]
and Lemma \ref{lem:beta-semic} gives
\[
 a_n:=(\beta_0^{\hbar_n}(c_n))^{\frac23}N_{f_n}^{-\frac23}  \to a:=(\beta_0^0(c))^{\frac23}N_f^{-\frac23}.
\]
Writing $T_f:=\frac{I_f}{3N_f}$, a direct calculation from \eqref{eq:beta-clas} gives
\[
 e^{-c}=\frac{N_f}{(2\pi T_f)^{\frac32}},  \quad  a=\frac1{2T_f}.
\]
Thus, for almost every $p$,
\[
 \calF^{\hbar_n}(f_n)(t,x,p)  \to  \frac{N_f}{(2\pi T_f)^{\frac32}}  e^{-|p-u|^2/(2T_f)}  =M(f)(t,x,p).
\]
Since we know that
\[
 \int_{\R^3}\calF^{\hbar_n}(f_n)\,\dpp=N_{f_n}\to N_f=\int_{\R^3} M(f)\,\dpp \quad \text{a.e.},
\]
we deduce from Lemma \ref{lem:pos-L1} that for a.e. $(t,x)$ such that $N_f(t,x)>0$:
\bq \label{eq:semi-pt-eql}
 \|\calF^{\hbar_n}(f_n)(t,x,\cdot)-M(f)(t,x,\cdot)\|_{L^1(\R^3_p)}\to 0.
\eq
Recalling that the vacuum set was already discussed at the beginning of Step 4, we conclude that \eqref{eq:semi-pt-eql} holds for a.e. $(t,x)$.

Now,
\[
 \int_{\R^3}\lt|  \calF^{\hbar_n}(f_n)(t,x,p)-M(f)(t,x,p)\rt|\,\dpp  \leq  N_{f_n}(t,x)+N_f(t,x).
\]
By \eqref{eq:semi-den-glo},
\[
 N_{f_n}+N_f\to2N_f  \quad\text{strongly in }L^1((0,T)\times\R^3_x).
\]
Combining this with the almost-everywhere convergence \eqref{eq:semi-pt-eql}, and using the same argument as in the proof of \eqref{eq:F-str-glo}, we obtain
\bq \label{eq:semi-F-str}
 \calF^{\hbar_n}(f_n)\to M(f)  \quad\text{strongly in }L^1((0,T)\times\R^3\times\R^3).
\eq

\emph{Step 5: passage to the mild equation.}
For every $n$,
\[
 f_n(t,x,p) =e^{-t}f_{0,n}(x-tp,p) +\int_0^te^{-(t-s)}\calF^{\hbar_n}(f_n) (s,x-(t-s)p,p)\,\ds.
\]
Define
\[
 \widetilde f(t,x,p) = e^{-t}f_0(x-tp,p) +\int_0^te^{-(t-s)}M(f)(s,x-(t-s)p,p)\,\ds.
\]
By subtracting the two equations, using straightforward change of variables, and next \eqref{eq:semi-ini-conv} and \eqref{eq:semi-F-str}, we deduce
\[
 \sup_{0\leq t\leq T}\|f_n(t)-\widetilde f(t)\|_{L^1(\R^3\times\R^3)} \leq\|f_{0,n}-f_0\|_{L^1(\R^3\times\R^3)} +\int_0^T\|\calF^{\hbar_n}(f_n)(s)-M(f)(s)\|_{L^1(\R^3\times\R^3)}\,\ds \to0.
\]
In other words we have proven \eqref{eq:semi-f-str}. Comparing with the weak limit \eqref{eq:semi-weak}, the uniqueness of distributional limits shows $\widetilde f=f$. Thus we have proven \eqref{eq:classical-mild}.
\end{proof}

We conclude the semiclassical analysis by verifying the particular case of the canonical approximation (truncating the initial datum with the Pauli bound). Note that the Pauli truncations $f^{\hbar}_0$ stated there converge strongly to the classical initial datum $f_0$ (thanks to the monotone convergence theorem), while preserving the required moment and entropy bounds uniformly in $\hbar$.

\begin{proof}[Proof of Corollary \ref{cor:canon}]
Fix a sequence $\hbar_n\downarrow0$ and set
\[
 f_{0,n}:=f_0^{\hbar_n}.
\]
The monotone convergence theorem gives
\[
 \|f_0^\hbar-f_0\|_{L^1(\R^3\times\R^3)}\to0.
\]
Moreover, using that $f_0^{\hbar_n} \le f_0$,
\[
 \sup_{0<\hbar\leq1}  \iint_{\R^3\times\R^3}  (1+|x|^2+|p|^2)f_0^\hbar\,\dx\dpp  \leq  \iint_{\R^3\times\R^3}  (1+|x|^2+|p|^2)f_0\,\dx\dpp.
\]
Since truncation occurs only on $\lt\{f_0>\frac1\hbar\rt\}$, where $\hbar^{-1}\geq1$, monotonicity of $z\mapsto z\log z$ on $[1,\infty)$ gives
\[
 f_0^\hbar|\log f_0^\hbar|  \leq  f_0|\log f_0|.
\]
Thus the uniform entropy bound also holds, and Theorem \ref{thm:semi} applies.
\end{proof}

%
%
%
%
%
%

\section*{Acknowledgments}
The work of Y.-P. Choi and S. Song was supported by NRF grants no. 2022R1A2C1002820 and RS-2024-00406821. The work of B.-H. Hwang was supported by the National Research Foundation of Korea(NRF) grant funded by the Korean goverment(MSIT) RS-2026-25475225.

%
%
%
%
%
%
\bibliographystyle{abbrv}
\bibliography{FDBGK}

@article{GH26,
  title={Global smooth solutions to the inhomogeneous {L}andau-{F}ermi-{D}irac equation},
  author={Golding, William and Henderson, Christopher},
  journal={arXiv:2608.31071},
}

@article {Yun10,
    AUTHOR = {Yun, Seok-Bae},
     TITLE = {Cauchy problem for the {B}oltzmann-{BGK} model near a global
              {M}axwellian},
   JOURNAL = {J. Math. Phys.},
  FJOURNAL = {Journal of Mathematical Physics},
    VOLUME = {51},
      YEAR = {2010},
    NUMBER = {12},
     PAGES = {123514, 24},
      ISSN = {0022-2488,1089-7658},
   MRCLASS = {82C40 (35Q20)},
  MRNUMBER = {2779616},
MRREVIEWER = {Cecil\ Pompiliu\ Gr\"unfeld},
       DOI = {10.1063/1.3516479},
       URL = {https://doi.org/10.1063/1.3516479},
}

@article {Yun15JDE,
    AUTHOR = {Yun, Seok-Bae},
     TITLE = {Classical solutions for the ellipsoidal {BGK} model with fixed
              collision frequency},
   JOURNAL = {J. Differential Equations},
  FJOURNAL = {Journal of Differential Equations},
    VOLUME = {259},
      YEAR = {2015},
    NUMBER = {11},
     PAGES = {6009--6037},
      ISSN = {0022-0396,1090-2732},
   MRCLASS = {35Q20 (35A09 35F25)},
  MRNUMBER = {3397316},
MRREVIEWER = {Cecil\ Pompiliu\ Gr\"unfeld},
       DOI = {10.1016/j.jde.2015.07.016},
       URL = {https://doi.org/10.1016/j.jde.2015.07.016},
}

@article {Yun15SIAM,
    AUTHOR = {Yun, Seok-Bae},
     TITLE = {Ellipsoidal {BGK} model near a global {M}axwellian},
   JOURNAL = {SIAM J. Math. Anal.},
  FJOURNAL = {SIAM Journal on Mathematical Analysis},
    VOLUME = {47},
      YEAR = {2015},
    NUMBER = {3},
     PAGES = {2324--2354},
      ISSN = {0036-1410,1095-7154},
   MRCLASS = {76P05 (35Q20 82C40)},
  MRNUMBER = {3357626},
MRREVIEWER = {Marzia\ Bisi},
       DOI = {10.1137/130932399},
       URL = {https://doi.org/10.1137/130932399},
}

@article {HY19,
    AUTHOR = {Hwang, Byung-Hoon and Yun, Seok-Bae},
     TITLE = {Ellipsoidal {BGK} model near a global {M}axwellian in the
              whole space},
   JOURNAL = {J. Math. Phys.},
  FJOURNAL = {Journal of Mathematical Physics},
    VOLUME = {60},
      YEAR = {2019},
    NUMBER = {7},
     PAGES = {071507, 28},
      ISSN = {0022-2488,1089-7658},
   MRCLASS = {35Q82 (35A09 35B40 35Q30)},
  MRNUMBER = {3981601},
       DOI = {10.1063/1.5017899},
       URL = {https://doi.org/10.1063/1.5017899},
}

@article {BKPY23,
    AUTHOR = {Bae, Gi-Chan and Klingenberg, Christian and Pirner, Marlies
              and Yun, Seok-Bae},
     TITLE = {B{GK} model for two-component gases near a global
              {M}axwellian},
   JOURNAL = {SIAM J. Math. Anal.},
  FJOURNAL = {SIAM Journal on Mathematical Analysis},
    VOLUME = {55},
      YEAR = {2023},
    NUMBER = {2},
     PAGES = {1007--1047},
      ISSN = {0036-1410,1095-7154},
   MRCLASS = {82C40 (35F16 35Q20 76P05)},
  MRNUMBER = {4579721},
MRREVIEWER = {Shuangqian\ Liu},
       DOI = {10.1137/22M1469535},
       URL = {https://doi.org/10.1137/22M1469535},
}

@article {BKLY26,
    AUTHOR = {Bae, Gi-Chan and Ko, Gyounghun and Lee, Donghyun and Yun,
              Seok-Bae},
     TITLE = {Large amplitude problem of {BGK} model: relaxation to
              quadratic nonlinearity},
   JOURNAL = {SIAM J. Math. Anal.},
  FJOURNAL = {SIAM Journal on Mathematical Analysis},
    VOLUME = {58},
      YEAR = {2026},
    NUMBER = {2},
     PAGES = {1530--1570},
      ISSN = {0036-1410,1095-7154},
   MRCLASS = {35Q20 (82C40 82D05)},
  MRNUMBER = {5057219},
       DOI = {10.1137/25M1774537},
       URL = {https://doi.org/10.1137/25M1774537},
}

@article {Per89,
    AUTHOR = {Perthame, B.},
     TITLE = {Global existence to the {BGK} model of {B}oltzmann equation},
   JOURNAL = {J. Differential Equations},
  FJOURNAL = {Journal of Differential Equations},
    VOLUME = {82},
      YEAR = {1989},
    NUMBER = {1},
     PAGES = {191--205},
      ISSN = {0022-0396,1090-2732},
   MRCLASS = {35Q99 (45K05 76P05 82B40)},
  MRNUMBER = {1023307},
       DOI = {10.1016/0022-0396(89)90173-3},
       URL = {https://doi.org/10.1016/0022-0396(89)90173-3},
}

@article {PP93,
    AUTHOR = {Perthame, B. and Pulvirenti, M.},
     TITLE = {Weighted {$L^\infty$} bounds and uniqueness for the
              {B}oltzmann {BGK} model},
   JOURNAL = {Arch. Rational Mech. Anal.},
  FJOURNAL = {Archive for Rational Mechanics and Analysis},
    VOLUME = {125},
      YEAR = {1993},
    NUMBER = {3},
     PAGES = {289--295},
      ISSN = {0003-9527},
   MRCLASS = {82C40 (76P05)},
  MRNUMBER = {1245074},
MRREVIEWER = {Carlo\ Cercignani},
       DOI = {10.1007/BF00383223},
       URL = {https://doi.org/10.1007/BF00383223},
}

@article {BY20,
    AUTHOR = {Bae, Gi-Chan and Yun, Seok-Bae},
     TITLE = {Quantum {BGK} model near a global {F}ermi-{D}irac
              distribution},
   JOURNAL = {SIAM J. Math. Anal.},
  FJOURNAL = {SIAM Journal on Mathematical Analysis},
    VOLUME = {52},
      YEAR = {2020},
    NUMBER = {3},
     PAGES = {2313--2352},
      ISSN = {0036-1410,1095-7154},
   MRCLASS = {35Q20 (76P05 82B40 82C40)},
  MRNUMBER = {4096124},
MRREVIEWER = {Luisa\ Arlotti},
       DOI = {10.1137/19M1270021},
       URL = {https://doi.org/10.1137/19M1270021},
}

@article {BY20S,
    AUTHOR = {Bae, Gi-Chan and Yun, Seok-Bae},
     TITLE = {Stationary quantum {BGK} model for bosons and fermions in a
              bounded interval},
   JOURNAL = {J. Stat. Phys.},
  FJOURNAL = {Journal of Statistical Physics},
    VOLUME = {178},
      YEAR = {2020},
    NUMBER = {4},
     PAGES = {845--868},
      ISSN = {0022-4715,1572-9613},
   MRCLASS = {35Q82 (82D05)},
  MRNUMBER = {4064205},
       DOI = {10.1007/s10955-019-02466-2},
       URL = {https://doi.org/10.1007/s10955-019-02466-2},
}

@article {Bra19,
    AUTHOR = {Braukhoff, Marcel},
     TITLE = {Semiconductor {B}oltzmann-{D}irac-{B}enney equation with a
              {BGK}-type collision operator: existence of solutions vs.
              ill-posedness},
   JOURNAL = {Kinet. Relat. Models},
  FJOURNAL = {Kinetic and Related Models},
    VOLUME = {12},
      YEAR = {2019},
    NUMBER = {2},
     PAGES = {445--482},
      ISSN = {1937-5093,1937-5077},
   MRCLASS = {82D37 (35Q20 82D10)},
  MRNUMBER = {3918275},
       DOI = {10.3934/krm.2019019},
       URL = {https://doi.org/10.3934/krm.2019019},
}

@article {Bra20,
    AUTHOR = {Braukhoff, Marcel},
     TITLE = {Global analytic solutions of the semiconductor
              {B}oltzmann-{D}irac-{B}enney equation with relaxation time
              approximation},
   JOURNAL = {Kinet. Relat. Models},
  FJOURNAL = {Kinetic and Related Models},
    VOLUME = {13},
      YEAR = {2020},
    NUMBER = {1},
     PAGES = {187--210},
      ISSN = {1937-5093,1937-5077},
   MRCLASS = {35Q20 (35F25 35Q83 82D37)},
  MRNUMBER = {4063920},
       DOI = {10.3934/krm.2020007},
       URL = {https://doi.org/10.3934/krm.2020007},
}

@article {BJY21,
    AUTHOR = {Bae, Gi-Chan and Jang, Jin Woo and Yun, Seok-Bae},
     TITLE = {The relativistic quantum {B}oltzmann equation near
              equilibrium},
   JOURNAL = {Arch. Ration. Mech. Anal.},
  FJOURNAL = {Archive for Rational Mechanics and Analysis},
    VOLUME = {240},
      YEAR = {2021},
    NUMBER = {3},
     PAGES = {1593--1644},
      ISSN = {0003-9527,1432-0673},
   MRCLASS = {35Q40 (82C40)},
  MRNUMBER = {4264953},
       DOI = {10.1007/s00205-021-01643-6},
       URL = {https://doi.org/10.1007/s00205-021-01643-6},
}

@article {OW22,
    AUTHOR = {Ouyang, Zhimeng and Wu, Lei},
     TITLE = {On the quantum {B}oltzmann equation near {M}axwellian and
              vacuum},
   JOURNAL = {J. Differential Equations},
  FJOURNAL = {Journal of Differential Equations},
    VOLUME = {316},
      YEAR = {2022},
     PAGES = {471--551},
      ISSN = {0022-0396,1090-2732},
   MRCLASS = {35Q82 (82C40)},
  MRNUMBER = {4377165},
       DOI = {10.1016/j.jde.2022.01.056},
       URL = {https://doi.org/10.1016/j.jde.2022.01.056},
}

@article {Dol94,
    AUTHOR = {Dolbeault, J.},
     TITLE = {Kinetic models and quantum effects: a modified {B}oltzmann
              equation for {F}ermi-{D}irac particles},
   JOURNAL = {Arch. Rational Mech. Anal.},
  FJOURNAL = {Archive for Rational Mechanics and Analysis},
    VOLUME = {127},
      YEAR = {1994},
    NUMBER = {2},
     PAGES = {101--131},
      ISSN = {0003-9527},
   MRCLASS = {82C40 (35Q99 76P05)},
  MRNUMBER = {1288807},
MRREVIEWER = {Carlo\ Cercignani},
       DOI = {10.1007/BF00377657},
       URL = {https://doi.org/10.1007/BF00377657},
}

@article {All10,
    AUTHOR = {Allemand, Thibaut},
     TITLE = {Existence and conservation laws for the
              {B}oltzmann-{F}ermi-{D}irac equation in a general domain},
   JOURNAL = {C. R. Math. Acad. Sci. Paris},
  FJOURNAL = {Comptes Rendus Math\'ematique. Acad\'emie des Sciences. Paris},
    VOLUME = {348},
      YEAR = {2010},
    NUMBER = {13-14},
     PAGES = {763--767},
      ISSN = {1631-073X,1778-3569},
   MRCLASS = {82C40 (35Q20 35Q40)},
  MRNUMBER = {2671157},
MRREVIEWER = {Laurent\ Desvillettes},
       DOI = {10.1016/j.crma.2010.06.015},
       URL = {https://doi.org/10.1016/j.crma.2010.06.015},
}

@article {Lu01,
    AUTHOR = {Lu, Xuguang},
     TITLE = {On spatially homogeneous solutions of a modified {B}oltzmann
              equation for {F}ermi-{D}irac particles},
   JOURNAL = {J. Statist. Phys.},
  FJOURNAL = {Journal of Statistical Physics},
    VOLUME = {105},
      YEAR = {2001},
    NUMBER = {1-2},
     PAGES = {353--388},
      ISSN = {0022-4715,1572-9613},
   MRCLASS = {82C40},
  MRNUMBER = {1861208},
MRREVIEWER = {Carlo\ Cercignani},
       DOI = {10.1023/A:1012282516668},
       URL = {https://doi.org/10.1023/A:1012282516668},
}

@article {LW03,
    AUTHOR = {Lu, Xuguang and Wennberg, Bernt},
     TITLE = {On stability and strong convergence for the spatially
              homogeneous {B}oltzmann equation for {F}ermi-{D}irac
              particles},
   JOURNAL = {Arch. Ration. Mech. Anal.},
  FJOURNAL = {Archive for Rational Mechanics and Analysis},
    VOLUME = {168},
      YEAR = {2003},
    NUMBER = {1},
     PAGES = {1--34},
      ISSN = {0003-9527,1432-0673},
   MRCLASS = {82C40 (82C10)},
  MRNUMBER = {2029003},
MRREVIEWER = {Carlo\ Cercignani},
       DOI = {10.1007/s00205-003-0247-8},
       URL = {https://doi.org/10.1007/s00205-003-0247-8},
}

@article {Lu08,
    AUTHOR = {Lu, Xuguang},
     TITLE = {On the {B}oltzmann equation for {F}ermi-{D}irac particles with
              very soft potentials: global existence of weak solutions},
   JOURNAL = {J. Differential Equations},
  FJOURNAL = {Journal of Differential Equations},
    VOLUME = {245},
      YEAR = {2008},
    NUMBER = {7},
     PAGES = {1705--1761},
      ISSN = {0022-0396,1090-2732},
   MRCLASS = {82C40 (35D05 35F25)},
  MRNUMBER = {2433484},
MRREVIEWER = {C\'edric\ Villani},
       DOI = {10.1016/j.jde.2008.06.028},
       URL = {https://doi.org/10.1016/j.jde.2008.06.028},
}

@article {JZ25,
    AUTHOR = {Jiang, Ning and Zhou, Kai},
     TITLE = {Global well-posedness of {B}oltzmann-{F}ermi-{D}irac equation
              for hard potential},
   JOURNAL = {Kinet. Relat. Models},
  FJOURNAL = {Kinetic and Related Models},
    VOLUME = {18},
      YEAR = {2025},
    NUMBER = {2},
     PAGES = {148--185},
      ISSN = {1937-5093,1937-5077},
   MRCLASS = {35Q20 (76P05 82B40 82C40)},
  MRNUMBER = {4874792},
MRREVIEWER = {Yuanjie\ Lei},
       DOI = {10.3934/krm.2024014},
       URL = {https://doi.org/10.3934/krm.2024014},
}

@article{AP25p,
  title={On the {B}oltzmann-{F}ermi-{D}irac equation for hard potential: global existence and uniqueness, {G}aussian lower bound, and moment estimates},
  author={An, Gayoung and Park, Sungbin},
  journal={arXiv:2511.02273},
}

@incollection {BL04,
    AUTHOR = {Bagland, V\'eronique and Lemou, Mohammed},
     TITLE = {Equilibrium states for the {L}andau-{F}ermi-{D}irac equation},
 BOOKTITLE = {Nonlocal elliptic and parabolic problems},
    SERIES = {Banach Center Publ.},
    VOLUME = {66},
     PAGES = {29--37},
 PUBLISHER = {Polish Acad. Sci. Inst. Math., Warsaw},
      YEAR = {2004},
   MRCLASS = {82C10 (35Q40 35R10 45K05)},
  MRNUMBER = {2143354},
MRREVIEWER = {Laurent\ Desvillettes},
       DOI = {10.4064/bc66-0-2},
       URL = {https://doi.org/10.4064/bc66-0-2},
}

@article {EMV05,
    AUTHOR = {Escobedo, Miguel and Mischler, St\'ephane and Valle, Manuel
              A.},
     TITLE = {Entropy maximisation problem for quantum relativistic
              particles},
   JOURNAL = {Bull. Soc. Math. France},
  FJOURNAL = {Bulletin de la Soci\'et\'e{} Math\'ematique de France},
    VOLUME = {133},
      YEAR = {2005},
    NUMBER = {1},
     PAGES = {87--120},
      ISSN = {0037-9484,2102-622X},
   MRCLASS = {82B40 (82C40)},
  MRNUMBER = {2145021},
MRREVIEWER = {C\'edric\ Villani},
       DOI = {10.24033/bsmf.2480},
       URL = {https://doi.org/10.24033/bsmf.2480},
}

@article {ABDL22,
    AUTHOR = {Alonso, R. and Bagland, V. and Desvillettes, L. and Lods, B.},
     TITLE = {About the {L}andau-{F}ermi-{D}irac equation with moderately
              soft potentials},
   JOURNAL = {Arch. Ration. Mech. Anal.},
  FJOURNAL = {Archive for Rational Mechanics and Analysis},
    VOLUME = {244},
      YEAR = {2022},
    NUMBER = {3},
     PAGES = {779--875},
      ISSN = {0003-9527,1432-0673},
   MRCLASS = {35Q82 (82C40)},
  MRNUMBER = {4419608},
       DOI = {10.1007/s00205-022-01779-z},
       URL = {https://doi.org/10.1007/s00205-022-01779-z},
}

@article {GGZ22,
    AUTHOR = {Golding, William and Gualdani, Maria Pia and Zamponi, Nicola},
     TITLE = {Existence of smooth solutions to the {L}andau-{F}ermi-{D}irac
              equation with {C}oulomb potential},
   JOURNAL = {Commun. Math. Sci.},
  FJOURNAL = {Communications in Mathematical Sciences},
    VOLUME = {20},
      YEAR = {2022},
    NUMBER = {8},
     PAGES = {2315--2365},
      ISSN = {1539-6746,1945-0796},
   MRCLASS = {35K59 (35K55 35P15 82C40 82D10)},
  MRNUMBER = {4521042},
       DOI = {10.4310/cms.2022.v20.n8.a7},
       URL = {https://doi.org/10.4310/cms.2022.v20.n8.a7},
}

@article{Sam24p,
  title={Global solutions to the {L}andau-{F}ermi-{D}irac equation},
  author={Sampaio, Paulo},
  journal={arXiv:2410.12681},
}

@article {Sam26,
    AUTHOR = {Sampaio, Paulo},
     TITLE = {On the semi-classical limit for the {L}andau-{F}ermi-{D}irac
              equation},
   JOURNAL = {Kinet. Relat. Models},
  FJOURNAL = {Kinetic and Related Models},
    VOLUME = {21},
      YEAR = {2026},
     PAGES = {104--139},
      ISSN = {1937-5093,1937-5077},
   MRCLASS = {35Q99 (82D10)},
  MRNUMBER = {5041646},
       DOI = {10.3934/krm.2026005},
       URL = {https://doi.org/10.3934/krm.2026005},
}

@article {GLPS88,
    AUTHOR = {Golse, Fran\c cois and Lions, Pierre-Louis and Perthame,
              Beno\^it and Sentis, R\'emi},
     TITLE = {Regularity of the moments of the solution of a transport
              equation},
   JOURNAL = {J. Funct. Anal.},
  FJOURNAL = {Journal of Functional Analysis},
    VOLUME = {76},
      YEAR = {1988},
    NUMBER = {1},
     PAGES = {110--125},
      ISSN = {0022-1236},
   MRCLASS = {35Q20 (35B99)},
  MRNUMBER = {923047},
       DOI = {10.1016/0022-1236(88)90051-1},
       URL = {https://doi.org/10.1016/0022-1236(88)90051-1},
}

@article {DLM91,
    AUTHOR = {DiPerna, R. J. and Lions, P.-L. and Meyer, Y.},
     TITLE = {{$L^p$} regularity of velocity averages},
   JOURNAL = {Ann. Inst. H. Poincar\'e{} C Anal. Non Lin\'eaire},
  FJOURNAL = {Annales de l'Institut Henri Poincar\'e{} C. Analyse Non
              Lin\'eaire},
    VOLUME = {8},
      YEAR = {1991},
    NUMBER = {3-4},
     PAGES = {271--287},
      ISSN = {0294-1449,1873-1430},
   MRCLASS = {35B65 (35Q99 82C70)},
  MRNUMBER = {1127927},
MRREVIEWER = {Benoit\ Perthame},
       DOI = {10.1016/S0294-1449(16)30264-5},
       URL = {https://doi.org/10.1016/S0294-1449(16)30264-5},
}

@article {GSR02,
    AUTHOR = {Golse, Fran\c cois and Saint-Raymond, Laure},
     TITLE = {Velocity averaging in {$L^1$} for the transport equation},
   JOURNAL = {C. R. Math. Acad. Sci. Paris},
  FJOURNAL = {Comptes Rendus Math\'ematique. Acad\'emie des Sciences. Paris},
    VOLUME = {334},
      YEAR = {2002},
    NUMBER = {7},
     PAGES = {557--562},
      ISSN = {1631-073X,1778-3569},
   MRCLASS = {35F20 (35Q35 82C70)},
  MRNUMBER = {1903763},
       DOI = {10.1016/S1631-073X(02)02302-6},
       URL = {https://doi.org/10.1016/S1631-073X(02)02302-6},
}

@article {KS25,
    AUTHOR = {Koo, Dowan and Song, Sihyun},
     TITLE = {Global mild solutions to a {BGK} model for barotropic gas
              dynamics},
   JOURNAL = {SIAM J. Math. Anal.},
  FJOURNAL = {SIAM Journal on Mathematical Analysis},
    VOLUME = {57},
      YEAR = {2025},
    NUMBER = {4},
     PAGES = {4137--4164},
      ISSN = {0036-1410,1095-7154},
   MRCLASS = {82C40 (35F25 76N15)},
  MRNUMBER = {4941924},
       DOI = {10.1137/24M1688047},
       URL = {https://doi.org/10.1137/24M1688047},
}

@article {DL89,
    AUTHOR = {DiPerna, R. J. and Lions, P.-L.},
     TITLE = {On the {C}auchy problem for {B}oltzmann equations: global
              existence and weak stability},
   JOURNAL = {Ann. of Math. (2)},
  FJOURNAL = {Annals of Mathematics. Second Series},
    VOLUME = {130},
      YEAR = {1989},
    NUMBER = {2},
     PAGES = {321--366},
      ISSN = {0003-486X,1939-8980},
   MRCLASS = {82A40 (35Q20 45K05 76P05)},
  MRNUMBER = {1014927},
MRREVIEWER = {Seiji\ Ukai},
       DOI = {10.2307/1971423},
       URL = {https://doi.org/10.2307/1971423},
}

@article {MP97,
    AUTHOR = {Mischler, S. and Perthame, B.},
     TITLE = {Boltzmann equation with infinite energy: renormalized
              solutions and distributional solutions for small initial data
              and initial data close to a {M}axwellian},
   JOURNAL = {SIAM J. Math. Anal.},
  FJOURNAL = {SIAM Journal on Mathematical Analysis},
    VOLUME = {28},
      YEAR = {1997},
    NUMBER = {5},
     PAGES = {1015--1027},
      ISSN = {0036-1410},
   MRCLASS = {35Q99 (35D05 76P05 82C70)},
  MRNUMBER = {1466666},
       DOI = {10.1137/S0036141096298102},
       URL = {https://doi.org/10.1137/S0036141096298102},
}

@article {Mis10,
    AUTHOR = {Mischler, St\'ephane},
     TITLE = {Kinetic equations with {M}axwell boundary conditions},
   JOURNAL = {Ann. Sci. \'Ec. Norm. Sup\'er. (4)},
  FJOURNAL = {Annales Scientifiques de l'\'Ecole Normale Sup\'erieure.
              Quatri\`eme S\'erie},
    VOLUME = {43},
      YEAR = {2010},
    NUMBER = {5},
     PAGES = {719--760},
      ISSN = {0012-9593,1873-2151},
   MRCLASS = {35F30 (35B35)},
  MRNUMBER = {2721875},
       DOI = {10.24033/asens.2132},
       URL = {https://doi.org/10.24033/asens.2132},
}

@article {Nou08,
    AUTHOR = {Nouri, Anne},
     TITLE = {An existence result for a quantum {BGK} model},
   JOURNAL = {Math. Comput. Modelling},
  FJOURNAL = {Mathematical and Computer Modelling},
    VOLUME = {47},
      YEAR = {2008},
    NUMBER = {3-4},
     PAGES = {515--529},
      ISSN = {0895-7177},
   MRCLASS = {82C40 (82C10)},
  MRNUMBER = {2378854},
       DOI = {10.1016/j.mcm.2007.05.002},
       URL = {https://doi.org/10.1016/j.mcm.2007.05.002},
}

@article {BKPY21,
    AUTHOR = {Bae, Gi-Chan and Klingenberg, Christian and Pirner, Marlies
              and Yun, Seok-Bae},
     TITLE = {B{GK} model of the multi-species {U}ehling-{U}hlenbeck
              equation},
   JOURNAL = {Kinet. Relat. Models},
  FJOURNAL = {Kinetic and Related Models},
    VOLUME = {14},
      YEAR = {2021},
    NUMBER = {1},
     PAGES = {25--44},
      ISSN = {1937-5093,1937-5077},
   MRCLASS = {82C40 (35Q20 76Y05 82C10)},
  MRNUMBER = {4206987},
MRREVIEWER = {Abdennebi\ Omrane},
       DOI = {10.3934/krm.2020047},
       URL = {https://doi.org/10.3934/krm.2020047},
}

@article{UU33,
  title={Transport phenomena in {E}instein-{B}ose and {F}ermi-{D}irac gases. {I}},
  author={Uehling, Edwin Albrecht and Uhlenbeck, GE},
  journal={Physical Review},
  volume={43},
  number={7},
  pages={552},
  year={1933},
  publisher={APS}
}

@article {HLP21,
    AUTHOR = {He, Ling-Bing and Lu, Xuguang and Pulvirenti, Mario},
     TITLE = {On semi-classical limit of spatially homogeneous quantum
              {B}oltzmann equation: weak convergence},
   JOURNAL = {Comm. Math. Phys.},
  FJOURNAL = {Communications in Mathematical Physics},
    VOLUME = {386},
      YEAR = {2021},
    NUMBER = {1},
     PAGES = {143--223},
      ISSN = {0010-3616,1432-0916},
   MRCLASS = {35Q20 (82C40 82D05)},
  MRNUMBER = {4287184},
       DOI = {10.1007/s00220-021-04029-7},
       URL = {https://doi.org/10.1007/s00220-021-04029-7},
}

@article {HLPZ24,
    AUTHOR = {He, Ling-Bing and Lu, Xuguang and Pulvirenti, Mario and Zhou,
              Yu-Long},
     TITLE = {On semi-classical limit of spatially homogeneous quantum
              {B}oltzmann equation: asymptotic expansion},
   JOURNAL = {Comm. Math. Phys.},
  FJOURNAL = {Communications in Mathematical Physics},
    VOLUME = {405},
      YEAR = {2024},
    NUMBER = {12},
     PAGES = {Paper No. 297, 51},
      ISSN = {0010-3616,1432-0916},
   MRCLASS = {35Q20 (35C20 82C40 82D05)},
  MRNUMBER = {4829560},
       DOI = {10.1007/s00220-024-05174-5},
       URL = {https://doi.org/10.1007/s00220-024-05174-5},
}

\end{document}